\documentclass{amsart}

\usepackage[latin1]{inputenc}
\usepackage[english]{babel}
\usepackage{indentfirst}
\usepackage{amssymb}
\usepackage{amsthm}
\usepackage{xcolor}
\usepackage[all]{xy}

\def\epsilon{\varepsilon}

\newtheorem{theorem}{Theorem}[section]
\newtheorem{proposition}[theorem]{Proposition}
\newtheorem{corollary}[theorem]{Corollary}
\newtheorem{lemma}[theorem]{Lemma}

\theoremstyle{definition}
\newtheorem{definition}[theorem]{Definition}

\newtheorem{remark}[theorem]{Remark}

\def\11{\textbf{$1$}}

\begin{document}
\title[$(p,q)$-regular operators]{$(p,q)$-regular operators between Banach lattices }

\author{E. A. S\'{a}nchez P\'{e}rez}
\address{E. A. S\'{a}nchez P\'{e}rez\\ Instituto Universitario de Matem\'atica Pura y Aplicada\\  Universitat Polit\`ecnica de Val\`encia\\ 46022 Valencia. Spain.}
\email{easancpe@mat.upv.es}

\author[P. Tradacete]{P. Tradacete}
\address{P. Tradacete\\Department of Mathematics\\ Universidad Carlos III de Madrid\\ 28911, Legan\'es, Madrid, Spain.}
\email{ptradace@math.uc3m.es}

\thanks{E.A. S\'anchez P\'erez gratefully acknowledges support of Spanish Ministerio de Econom\'{\i}a, Industria y Competitividad under project MTM2016-77054-C2-1-P. P. Tradacete gratefully acknowledges support of Spanish Ministerio de Econom\'{\i}a, Industria y Competitividad through grants MTM2016-76808-P and MTM2016-75196-P, and Grupo UCM 910346.}

\maketitle

\begin{abstract}
We study the class of $(p,q)$-regular operators between quasi-Banach lattices. In particular, a representation of this class as the dual of a certain tensor norm for Banach lattices is given. We also provide some  factorization results for $(p,q)$-regular operators yielding new Marcinkiewicz-Zygmund type inequalities for Banach function spaces.  An extension theorem for $(q, \infty)$-regular operators defined on a subspace of $L_q$ is also given.
\end{abstract}


\section{Introduction}

This paper is devoted to study operators between Banach or quasi-Banach lattices satisfying estimates of the form
$$
\bigg\|\bigg(\sum_{i=1}^n|Tx_i|^p\bigg)^{\frac1p}\bigg\|\leq K \bigg\|\bigg(\sum_{i=1}^n|x_i|^q\bigg)^{\frac1q}\bigg\|,
$$
for every choice of vectors $\{x_i\}_{i=1}^n$. The operators for which this inequality holds are called $(p,q)$-regular, and as far as we know were introduced by A. Bukhvalov in \cite{B1} in connection with the interpolation of Banach lattices (see also \cite{B2}). The aim of this note is to make a systematic study of the class of $(p,q)$-regular operators.

It should be noted that the notion of regular operator has different usages in the literature: these can refer to operators that can be written as a difference of positive operators (cf. \cite{AB}), or for the terminology used in \cite{Pisier}, these refer to operators $T:X\rightarrow Y$ for which there is $K>0$ such that
$$
\bigg\|\bigvee_{i=1}^n|Tx_i|\bigg\|\leq K \bigg\|\bigvee_{i=1}^n|x_i|\bigg\|
$$
for every $\{x_i\}_{i=1}^n\subset X$. The latter correspond in the current terminology to $(\infty,\infty)$-regular operators.

There has been a considerable interest in the literature to determine conditions under which every operator between two Banach lattices is $(r,r)$-regular (or for brevity $r$-regular). In particular, an application of Grothendieck's inequality due to J. L. Krivine \cite{Krivine} (see also \cite[Theorem 1.f.14]{LT2}) yields that for any Banach lattices $X,Y$, every bounded linear operator $T:X\rightarrow Y$ is $2$-regular. This fact has been later extended by N. J. Kalton to quite a large family of quasi-Banach lattices \cite{kalton}, and is also related to the complexification constants of an operator (cf. \cite{GM}). On the other hand, it is also known that positive operators among Banach lattices are always $(p,q)$-regular for $q \le p$. The case of $(\infty,1)$-regular operators, which contain all $(p,q)$-regular operators, have been recently shown to have good interpolation properties with respect to the Calder\'on-Lozanovskii construction \cite{RT2}.

In the particular case of operators between $L_p$ spaces, the question whether every operator is $r$-regular can be traced back to classical works of R. Paley, J. Marcinkiewicz, A. Zygmund and S. Kwapien, and has been completely settled by A. Defant and M. Junge in \cite{DJ} (see also the references therein). In particular, in that paper, the authors characterize the triples $(p,q,r)$ for which every operator $T:L_p\rightarrow L_q$ is $r$-regular, even providing quantitative versions and asymptotic estimates of the constants involved in some cases. In this paper, we will analyze the analogous situation concerning $(p,q)$-regular operators.

It goes without saying that there is also a natural interplay between $(p,q)$-regularity and summability properties. This connection stems from the fact that for a Banach lattice $X$ and $\{x_i\}_{i=1}^n\subset X$ it holds that
$$
\sup_{x^* \in B_{X^*}} \Big( \sum |\langle x_i, x^* \rangle|^p \Big)^{1/p} \le \Big\| \Big( \sum | x_i|^p \Big)^{1/p} \Big\|_X.
$$
In particular, the above inequality yields that every lattice $(p,q)$-summing operator is $(p,q)$-regular (see \cite{Danet, NS} concerning lattice summing operators). The natural connection with convexity and concavity will also be explored.

The paper is organized as follows: after Section \ref{prelim}, where a preliminary discussion about the basics on $(p,q)$-regular operators is given, in Section \ref{tensor} we present several facts for this class of operators in the framework of Laprest\'e tensor norms (see \cite{deflo}). In particular, this allows to represent the class of $(p,q)$-regular operators as the dual of a certain tensor product in a standard way. It must be noted though that this point of view has been historically considered for Banach spaces, or locally convex vector spaces, but it is not so common for the case of Banach lattices. This approach is not to be confused with that of Banach lattice tensor norms, which studies the conditions under which a particular topology in a tensor product of Banach lattices becomes itself a Banach lattice (see the founding paper of D. Fremlin \cite{Fremlin} and related recent work of A. Schep in \cite{Schep}).

Next part of the paper, Section \ref{subsecFacBan}, is devoted to the peculiarities of $(p,q)$-regular operators between $L_r$ spaces. We first study the factorization properties of these operators in terms of the Maurey-Rosenthal theory (see Theorems \ref{thA} and \ref{t:factor2}) in order to provide a characterization of a specific class of operators factoring through $L_r$-spaces. As a main result of this section, we give a new class of Marcinkiewicz-Zygmund inequalities involving norms of general Banach function lattices.

The paper is finished by showing the extension properties of the $(\infty,q)$-regular operators defined on a subspace of a Banach lattice (Theorem \ref{t:extension}), which provide a version of a result of G. Pisier on extension of $\infty$-regular operators (see \cite{Pisier}).

We use standard terminology from Banach spaces, Banach lattices and operator theory. For any unexplained notion the reader is referred to the monographs \cite{AB, deflo, LT2}.

\section{Definitions and preliminaries}\label{prelim}

Suppose $X$ is a quasi-Banach lattice of measurable functions on a measure space $(\Omega,\Sigma,\mu)$. Given $\{x_i\}_{i=1}^n\subset X$ and $p\in(0,\infty)$, expressions of the form $(\sum_{i=1}^n|x_i|^p)^{1/p}$ can be defined pointwise ($\mu$-almost everywhere). Although Krivine's functional calculus (\cite[Theorem 1.d.1]{LT2}, see also \cite{Popa} for the non-locally convex setting) gives a meaning to this kind of expressions for abstract quasi-Banach lattices, for most applications we will only be concerned with the case of measurable functions.

\begin{definition}\label{d:pqregular}
Given quasi-Banach lattices $X,Y$, and $0< p,q< \infty$ a linear map $T:X\rightarrow Y$ is \emph{$(p,q)$-regular} if there is a constant $K>0$ such that for every $\{x_i\}_{i=1}^n\subset X,$
$$
\bigg\|\bigg(\sum_{i=1}^n|Tx_i|^p\bigg)^{\frac1p}\bigg\|\leq K \bigg\|\bigg(\sum_{i=1}^n|x_i|^q\bigg)^{\frac1q}\bigg\|.
$$
Similarly, $T$ is $(p,\infty)$-regular (respectively, $(\infty,q)$-regular) when
$$
\bigg\|\bigg(\sum_{i=1}^n|Tx_i|^p\bigg)^{\frac1p}\bigg\|\leq K \bigg\|\bigvee_{i=1}^n|x_i|\bigg\|. \hspace{1cm}\Bigg( \textrm{resp. } \bigg\|\bigvee_{i=1}^n|Tx_i|\bigg\|\leq K  \bigg\|\bigg(\sum_{i=1}^n|x_i|^q\bigg)^{\frac1q}\bigg\|.\Bigg)
$$
\end{definition}

For simplicity, when $p=q$ we will say that $T$ is $p$-regular. With this notation, the well-known notion of regular operator (cf. \cite{Pisier}) corresponds to the case of $\infty$-regular operator.

We will write $R_{p,q}(X,Y)$ for the space of $(p,q)$-regular operators between $X$ and $Y$. We will denote by $\rho_{p,q}(T)$ the smallest $K>0$ for which the inequalities appearing in Definition \ref{d:pqregular}  hold for arbitrary elements in $X$. The following facts are straightforward:

\begin{proposition}\label{p:pqregbasic}\
\begin{enumerate}
\item Every $(p,q)$-regular linear map $T$ is bounded with $\|T\|\leq\rho_{p,q}(T)$.
\item Let $p_1\geq p$ and $q_1\leq q$. If $T$ is $(p,q)$-regular, then $T$ is $(p_1,q_1)$-regular with $\rho_{p_1,q_1}(T)\leq\rho_{p,q}(T)$.
\end{enumerate}
\end{proposition}

Given $\{x_i\}_{i=1}^n\subset X$, for $p\geq 1$, taking $1/p+1/p'=1$ we can write
$$
\bigg(\sum_{i=1}^n|x_i|^p\bigg)^{\frac1p}=\sup\bigg\{\sum_{i=1}^n a_i x_i:\sum_{i=1}^n |a_i|^{p'}\leq1\bigg\},
$$
(the supremum being taken in the sense of order in the lattice $X$). In particular when $T$ is a positive operator, we have the inequalities
$$
\bigg(\sum_{i=1}^n|Tx_i|^p\bigg)^{\frac1p}=\sup_{\underset{i=1}{\overset{n}\sum} |a_i|^{p'}\leq1}T\Big(\sum_{i=1}^n a_i x_i\Big)\leq T\bigg(\sup_{\underset{i=1}{\overset{n}\sum} |a_i|^{p'}\leq1}\sum_{i=1}^n a_i x_i\bigg) =T\bigg(\sum_{i=1}^n|x_i|^p\bigg)^{\frac1p}
$$
More generally, if the modulus of an operator $T:X\rightarrow Y$ exists (cf. \cite[Chapter 1]{AB}), then it is $p$-regular for every $1\leq p\leq \infty$, and $\rho_{p,p}(T)\leq\||T|\|$. In particular, this happens when $Y$ is a Dedekind complete Banach lattice and $T$ can be written as a difference of two positive operators. Conversely, suppose $Y$ is complemented by a positive projection in its bidual $Y''$, then every $1$-regular operator $T:X\rightarrow Y$ can be written as a difference of two positive operators \cite[p. 307]{Kusraev}.

The definition of a $(p,q)$-regular operator suggests a connection with convexity and concavity. Indeed, recall that an operator $T:X\rightarrow Y$ is $(p,q)$-concave (cf. \cite[p. 330]{DJT}) whenever  there is a constant $C>0$ such that for every $\{x_i\}_{i=1}^n\subset X$
$$
\bigg(\sum_{i=1}^n\|Tx_i\|^p\bigg)^{\frac1p}\leq C \bigg\|\bigg(\sum_{i=1}^n|x_i|^q\bigg)^{\frac1q}\bigg\|.
$$
It is straightforward to check that, for $p\geq q$, if $T:X\rightarrow Y$ is $(p,q)$-concave and $S:Y\rightarrow Z$ is $p$-convex, then $ST$ is $(p,q)$-regular with $\rho_{p,q}(ST)\leq M^{(p)}(S) K_{p,q}(T)$ (where, $M^{(p)}(S)$ and $K_{p,q}(T)$ denote, respectively, the $p$-convexity constant of $S$ and the $(p,q)$-concavity constant of $T$, cf. \cite[1.d.3]{LT2}).

We mentioned above that for a large class of quasi-Banach lattices every bounded linear operator is $2$-regular. In fact, an application of Grothendieck's inequality due to J. L. Krivine \cite{Krivine} (see also \cite[Theorem 1.f.14]{LT2}) yields that for Banach lattices $X,Y$, every bounded linear operator $T:X\rightarrow Y$ is $2$-regular with $\rho_{2,2}(T)\leq K_G\|T\|$, where $K_G$ denotes Grothendieck's constant.

This fact was extended by N. J. Kalton to $L$-convex quasi-Banach lattices \cite{kalton}. Recall that a quasi-Banach lattice $X$ is $L$-convex whenever its order intervals are uniformly locally convex, that is, whenever there exists $0<\varepsilon<1$ so that if $u\in X_+$ with $\|u\|=1$ and $0\leq x_i\leq u$ (for $i=1,\ldots,n$) satisfy
$$
\frac1n(x_1+\ldots+x_n)\geq(1-\varepsilon)u,
$$
then
$$
\max_{1\leq i\leq n}\|x_i\|\geq\varepsilon.
$$

This class includes every quasi-Banach lattice which is the $p$-concavification of a Banach lattice (for instance, $L_p$, $\Lambda(W,p)$ and  $L_{p,\infty}$ for $0<p<\infty$). Kalton's result states that if $Y$ is an $L$-convex quasi-Banach lattice, then for every quasi-Banach lattice $X$, every operator $T:X\rightarrow Y$ is $2$-regular \cite[Theorem 3.3]{kalton}.

\begin{proposition}
Given quasi-Banach lattices $X,Y$ and $0<p,q\leq\infty$, the space $(R_{p,q}(X,Y),\rho_{p,q}(\cdot))$ is a quasi-Banach space.
\end{proposition}

\begin{proof}
This is straightforward. For completeness, just note that $\|T\|\leq\rho_{p,q}(T)$.
\end{proof}

\begin{proposition} Let $X,Y$ be quasi-Banach lattices.
\begin{itemize}
\item[(i)]
Let $0< p < q \le \infty$. Then $R_{p,q}(X,Y)=\{0\}$.
\item[(ii)] Suppose $Y$ is $L$-convex, then for every $p\geq 2\geq q$ we have $R_{p,q}(X,Y) = L(X,Y)$.
\item[(iii)] For $0< q \le p$, if $X$ is $q$-concave and $Y$ is $p$-convex, then $R_{p,q}(X,Y)= L(X,Y).$
\end{itemize}
\end{proposition}

\begin{proof}
(i) Let $T \in R_{p,q}(X,Y)$. Suppose there is $x \in X$ such that $Tx \ne 0$. Then for each $n \in \mathbb N$,
$$
n^\frac1p \| Tx\| = \|(\sum_{i=1}^n |T x|^p)^\frac1p \| \le \rho_{p,q}(T)  \| (\sum_{i=1}^{n} |x|^q )^\frac1q \| = \rho_{p,q}(T) \, n^\frac1q \|x\|.
$$
Since this is impossible for large $n$, we have that $T=0$.

(ii) This follows from \cite[Theorem 3.3]{kalton} and Proposition \ref{p:pqregbasic}.

(iii) Let $T:X\rightarrow Y$ be an operator. For $(x_i)_{i=1}^n\subset X$ we have
\begin{align*}
\|(\sum_{i=1}^n |T x_i|^p)^{\frac1p} \| &\le M^{(p)}(Y)(\sum_{i=1}^n \|T x_i\|^p)^{\frac1p} \le  M^{(p)}(Y)(\sum_{i=1}^n \|T x_i\|^q)^{\frac1q}\\
&\le  M^{(p)}(Y)\|T\| \, (\sum_{i=1}^n \|x_i\|^q)^{\frac1q}\le  M^{(p)}(Y)\|T\| M_{(q)}(X)\, \|(\sum_{i=1}^n |x_i|^q)^\frac1q  \|.
\end{align*}

\end{proof}

In the rest of the section we characterize $(p,q)$-regularity of operators in terms of the bilinear maps defined by them. Some of the results presented here are well-known; but we include them here in a unified way and with complete proofs for the aim of completeness. We first analyze a special type of duality in K\"othe-Bochner spaces, which will be done in what follows.

Recall that for a Banach lattice $X$ and $1\leq p< \infty$, $X(\ell_p)$ is the closed subspace of sequences $x=(x_n)_{n\in\mathbb N}\subset X$ for which
$$
\|x\|_{X(\ell_p)}=\sup_k\Big\|\Big(\sum_{n=1}^k|x_n|^p\Big)^{\frac1p}\Big\|<\infty,
$$
and which is spanned by the eventually null sequences. Similarly, $X(\ell_\infty)$ corresponds to the space of those sequences with
$$
\|x\|_{X(\ell_\infty)}=\sup_k\Big\|\bigvee_{n=1}^k|x_n|\Big\|<\infty.
$$
These are the natural generalization of Bochner (or K\"othe-Bochner) spaces for abstract Banach lattices. Indeed, let $X$ be an order continuous quasi-Banach function space over the measure space $(\Omega,\Sigma,\mu)$ and let $0<r \le \infty$. The K\"othe-Bochner space $X(\Omega,\Sigma,\mu; \ell^r)$ is defined to be the space of strongly $\Sigma$-measurable functions $\phi:\Omega \to \ell^r$ with the quasi-norm given by
$$
\|\phi\|_{X(\Omega,\Sigma,\mu;\ell^r)}:= \big\| \phi_{\|\cdot\|_r} \big\|_{X}, 
$$
where $\phi_{\|\cdot\|_r}:\Omega\to\mathbb R$ is given by $\omega\in \Omega \mapsto \|\phi(\omega)\|_{\ell^r}$.

\begin{lemma} \label{dense}
Let $X$ be an order continuous quasi-Banach function space over $(\Omega,\Sigma,\mu)$, $1\leq r\leq \infty$ and let $(e_i)_{i\in\mathbb N}$ denote the unit basis of $\ell_r$. For any $(x_i)_{i=1}^n \subset  X$ we have
\begin{itemize}

\item[(i)] the function $\phi:\Omega\to\ell^r$ given by $\phi(\omega)=\sum_{i=1}^n x_i(\omega)e_i$ belongs to $X(\Omega,\Sigma,\mu;\ell^r)$ with
$$
\|\phi\|_{X(\Omega,\Sigma,\mu; \ell^r)}:= \big\| \big( \sum_{i=1}^n |x_i|^r \big)^{\frac 1r} \big\|_{X},
$$

\item[(ii)] the set $X(\Omega,\Sigma,\mu; \ell^r)_0$ of all the functions defined in this way is dense in $X(\Omega,\Sigma,\mu;\ell^r)$.

\end{itemize}

\end{lemma}
\begin{proof}
(i) Since $X$ is order continuous, each function $x_i$ can be approximated by simple functions. Therefore, as the sequence is finite, a direct calculation shows that there is a sequence of $\ell^r$-valued simple functions converging in the quasi-norm of $X(\Omega,\Sigma,\mu; \ell^r)$ to $\phi$, and also that there is a subsequence of it that converges $\mu$-almost everywhere. Therefore, $\phi$ is strongly measurable. The formula for the quasi-norm is just the definition of the quasi-norm in the K\"othe-Bochner space.

(ii) By the order continuity of $X$, it can be  easily seen using also that the functions are strongly measurable that vector valued simple functions are dense in $X(\Omega,\Sigma,\mu;\ell^r)$.

Indeed, take  a sequence of simple functions $(x_i)$ converging $\mu$-a.e. to a function $x \in X(\Omega,\Sigma,\mu; \ell^r)$. We can choose a sequence $(y_i)$ such that for every $i$, $\|x(w)-y_{i+1} (w)\|_X \le \|x(w)- y_i(w)\|_X$ holds $\mu$-a.e. In order to see this, just take $x_1=y_1$ and consider the measurable set $A_2:=\{w| \|x(w)-x_1(w)\| \le \|x(w) -x_2(w)\| \}$,
and define the simple function $y_2= y_1 \chi_{A_2}  + x_2 \chi_{A_2^c}.$ Now, define $y_3$ in the same way using $y_2$ and $x_3$, and so on. The resulting sequence satisfies the requirement.
Now, we have that the real valued functions $\tau_i(w)=\|x(w)- y_i(w)\|$ are non-negative, decreasing and converge to $0$ $\mu$-a.e. The order continuity of $X(\mu)$ gives then that $\lim_i \tau_i=0$ in the norm of $X(\mu)$, that is, $\lim_i y_i =x$ in $X(\Omega,\Sigma,\mu;\ell^r)$.

 So it is enough to prove that there is a sequence of functions as in (i) converging to every simple function like $\psi=\sum_{i=1}^m u_i \chi_{A_i}$ in $X(\Omega,\Sigma,\mu; \ell^r)$, with $u_i=\sum_j \lambda_{i,j}e_j\in\ell_r$. In order to see that, note that for every $k$ we can write $\psi$ as
$$
\psi= \sum_{i=1}^m (\sum_{j=1}^k \lambda_{i,j} e_j) \cdot \chi_{A_i}+
\sum_{i=1}^m ( \sum_{j>k} \lambda_{i,j} e_j) \cdot \chi_{A_i}.
$$
Clearly, the first member in this sum, say $\psi_k= \sum_{i=1}^m (\sum_{j=1}^k \lambda_{i,j} e_j) \cdot \chi_{A_i}$, belongs to the above family. Thus, it is enough to check that the second member in the sum converges to $0$ in the quasi-norm, and so $\psi_k \to_k \psi$. Indeed,
$$
\big\| \sum_{i=1}^m (\sum_{j>k} \lambda_{i,j} e_j) \cdot  \chi_{A_i} \big\|_{X(\Omega,\Sigma,\mu;\ell^r)} \le
\sum_{i=1}^m ( \sum_{j= k+1}^\infty |\lambda_{i,j}|^r )^\frac 1r \| \chi_{A_i}\|_{X} \to_k 0.
$$
\end{proof}

Let $1 \le r, p,s \le \infty$ such that $1/r=1/p+1/s$, and $X$ a Banach lattice with dual $X'$. A similar argument as that of \cite[Proposition 1.d.2]{LT2} yields that for any $\{x_i\}_{i=1}^n\subset X$, $\{x'_i\}_{i=1}^n\subset X'$ it holds that
\begin{equation}\label{eq:genholder}
\Big(\sum_{i=1}^n|\langle x'_i,x_i\rangle|^r\Big)^{\frac1r}\leq\Big\langle \Big(\sum_{i=1}^n|x'_i|^s\Big)^{\frac1s}, \Big(\sum_{i=1}^n|x_i|^p\Big)^{\frac1p}\Big\rangle.
\end{equation}

For $r=1$ this inequality is sharp, in the sense of \cite[7.2.2 (2)]{Kusraev}. For $r>1$ this need not be the case, but in the K\"othe-Bochner setting there is an improvement which will be the key for our purposes. This can be understood as the consequence of some ``duality in the range" between norms of vector valued function spaces.

We will consider the ($\mu$-almost everywhere) pointwise product of vector valued functions as follows: for $\phi \in X(\Omega,\Sigma,\mu;\ell^p)$ and $\psi \in X'(\Omega,\Sigma,\mu; \ell^s)$,
\begin{equation}\label{eq:pointwiseproduct}
(\phi \cdot \psi) \, (\omega):= \sum_{i=1}^\infty \langle \phi(\omega),e_i \rangle \langle \psi(\omega),e_i \rangle \, e_i.
\end{equation}

H\"older's inequality in the range of the function yields that that $ (\phi \cdot \psi) \, (w) \in \ell^r$ $\mu$-almost everywhere. In fact we have the following:

\begin{lemma} \label{holder} (H\"older's inequality for $\ell^r$-valued K\"othe-Bochner functions)
Let $X$ be an order continuous Banach function space over $(\Omega, \Sigma, \mu)$, and $1 \le r\leq p,s \le \infty$ such that $1/r=1/p+1/s$. If $\phi \in X(\Omega, \Sigma, \mu;\ell^p)$ and $\psi \in X'(\Omega, \Sigma, \mu; \ell^s)$, then
$$
\big\| \phi \cdot \psi \big\|_{L^1(\Omega, \Sigma, \mu; \ell^r)} \le \big\| \phi \big\|_{X(\Omega, \Sigma, \mu;\ell^p)} \cdot
\big\| \psi \big\|_{X'(\Omega, \Sigma, \mu;\ell^s)}.
$$
\end{lemma}

\begin{proof}
The order continuity of $X$ gives the representation of the duality by means of the integral with respect to $\mu$, which can be taken as a probability measure (cf. \cite[Theorem 1.b.14]{LT2}.
Let $\phi$ and $\psi$ as in the statement. Then
\begin{align*}
\| \phi \cdot \psi \|_{L^1(\Omega, \Sigma, \mu; \ell^r)} & =
\int \big( \sum_{i=1}^\infty |\langle \phi(w),e_i \rangle \langle \psi(w), e_i \rangle|^r \big)^\frac 1r d \mu\\
&\le
\int  \, \big( \sum_{i=1}^\infty |\langle \phi(w), e_i \rangle|^p \big)^\frac 1p \cdot
 \big( \sum_{i=1}^\infty | \langle \psi(w), e_i \rangle|^s \big)^\frac 1s \,
 d \mu\\
&= \int \big\| \phi(w) \big\|_{\ell^p} \cdot
\big\| \psi(w) \big\|_{\ell^s} d \mu\\
&\le \big\| \phi \big\|_{X(\Omega, \Sigma, \mu;\ell^p)} \cdot
\big\| \psi \big\|_{X'(\Omega, \Sigma, \mu;\ell^s)}.
\end{align*}
\end{proof}

\begin{lemma} \label{dual}
Let $X$ be an order continuous Banach function space over $(\Omega, \Sigma, \mu)$ and consider $1 \le r \le p,s \le \infty$ satisfying that $1/r=1/p+1/s$. Given $\{x_i\}_{i=1}^n\subset X$ let $\phi= \sum_{i=1}^n x_i e_i \in X(\Omega, \Sigma, \mu;\ell^p)$. Then
$$
\| \phi \|_{X(\Omega, \Sigma, \mu;\ell^p)} = \sup_{\psi \in B_{X'(\Omega, \Sigma, \mu;\ell^s)}} \big\| \phi \cdot \psi \big\|_{L^1(\Omega, \Sigma, \mu; \ell^r)}.
$$
Actually, the functions in $B_{X'(\mu,\ell^s)}$ for the computation of the norm can be taken of the form $\sum_{i=1}^n x'_i e_i \in X'(\Omega, \Sigma, \mu;\ell^s)$, that is
$$
\big\| \big( \sum_{i=1}^n |x_i |^p \big)^\frac 1p \big\|_{X} =
\sup \Big\{ \big\| \big( \sum_{i=1}^n |x_i \cdot x'_i|^r \big)^\frac 1r \big\|_{L^1(\mu)}  :
\big\| \big( \sum_{i=1}^n |x'_i|^s \big)^\frac 1s \big\|_{X'} \le 1 \Big\}.
$$
\end{lemma}

\begin{proof}
The inequality $`` \ge "$ is a consequence of Lemma \ref{holder}. For the equality, consider $\{x_i\}_{i=1}^n\subset X$. By the Hahn-Banach Theorem, we can take $x' \in B_{X'}$ such that
$$
\big\| \big( \sum_{i=1}^n |x_i|^p \big)^\frac 1p \big\|_{X} = \langle  x',  \big( \sum_{i=1}^n |x_i|^p \big)^\frac 1p \rangle.
$$
For $i=1,\ldots, n$ let
$$
x'_i(\omega) = \left\{
\begin{array}{ccc}
 \frac{|x_i(\omega)|^{(p-r)/r} \, x'(\omega) }{ \big( \underset{i=1}{\overset{n}\sum} |x_i(\omega)|^p \big)^{1/s} }  &   & \textrm{if }\sum_{i=1}^n |x_i(\omega)|^p\neq0, \\
  &   &   \\
 0 &   &   \textrm{otherwise}.
\end{array}
\right.
$$
Note that since
$$
|x_i(\omega)|^{(p-r)/r}\leq \big( \sum_{i=1}^n |x_i(\omega)|^p \big)^\frac 1s,
$$
it follows that $x'_i\in X'$ for $i=1,\ldots, n$. Note also that
$$
\int \big( \sum_{i=1}^n |x_i \cdot x'_i|^r \big)^\frac 1r d\mu =
\int \big( \sum_{i=1}^n |x_i|^p \big)^{\frac 1p} x' d \mu
= \big\| \big( \sum_{i=1}^n |x_i|^p \big)^\frac 1p \big\|_{X}.
$$
Moreover, taking into account that
$$
\big(\frac{p-r}{r} \big) s= \frac{sp}{r} -s = p+s-s=p,
$$
we obtain
\begin{align*}
\big( \sum_{i=1}^n |x'_i|^s \big)^\frac 1s   = \Big( \sum_{i=1}^n \Big|\frac{|x_i|^{(p-r)/r} \, x' }{ \big( \sum_{i=1}^n |x_i|^p \big)^\frac 1s } \Big|^s \Big)^\frac 1s = x' \in B_{(X(\mu))'}.
\end{align*}
This proves the result.
\end{proof}

Given order continuous Banach function spaces $X$, $Y$ over $(\Omega, \Sigma, \mu)$ and $(\Omega', \Sigma', \nu)$ respectively, a linear operator $T:X \to Y$ and $1/r \le 1/q+1/s$, we can define a bilinear operator
$$
P_T:X(\Omega, \Sigma, \mu; \ell^q)_0 \times Y'(\Omega', \Sigma', \nu; \ell^s)_0 \to L^1(\nu, \ell^r)
$$
by means of the product defined in \eqref{eq:pointwiseproduct} for vector valued functions, as follows:
$$
P_T(\sum_{i=1}^n x_i e_i, \sum_{i=1}^n y'_i e_i)(\omega'):=
(\sum_{i=1}^n [Tx_i](\omega') e_i) \cdot ( \sum_{i=1}^n y'_i(\omega') e_i) =
\sum_{i=1}^n [Tx_i] (\omega') \cdot y'_i(\omega') e_i.
$$

The continuity of such a bilinear map is equivalent to the existence of a constant $C>0$ such that
$$
\big\| \sum_{i=1}^n Tx_i \, y'_i e_i \big\|_{L^1(\mu,\ell^r)}
\le C \, \big\|  \sum_{i=1}^n x_i  e_i  \big\| _{X(\Omega, \Sigma, \mu; \ell^q)} \cdot \big\|  \sum_{i=1}^n  y'_i e_i  \big\|_{Y'(\Omega', \Sigma', \nu; \ell^s)}.
$$
In this case, by Lemma \ref{dense}, $P_T$ uniquely extends to a continuous bilinear map on the space $X(\Omega, \Sigma, \mu; \ell^q)\times Y'(\Omega', \Sigma', \nu; \ell^s)$.

\begin{proposition}  \label{formulabil}
Let $X$, $Y$ be order continuous Banach function spaces over $(\Omega, \Sigma, \mu)$ and $(\Omega', \Sigma', \nu)$ respectively, and consider an operator $T:X \to Y$. Let $1/r=1/p+1/s$ and $q \le p$. The following assertions are equivalent.
\begin{itemize}
\item[(i)] $T$ is $(p,q)$-regular.
\item[(ii)] The bilinear map $P_T$ is continuous from $ X(\Omega, \Sigma, \mu; \ell^q)_0 \times Y'((\Omega', \Sigma', \nu; \ell^s)_0$  to $L^1(\nu, \ell^r)$.
\end{itemize}
Moreover, the continuity constant of the bilinear map $P_T$ equals $\rho_{p,q}(T)$.
\end{proposition}

\begin{proof}
(i) $\Rightarrow$ (ii): Let $\{x_i\}_{i=1}^n\subset X$ and $\{y'_i\}\subset Y'$. By Lemma \ref{holder} we have
\begin{align*}
\int \big( \sum_{i=1}^n |Tx_i y'_i|^r \big)^\frac 1r \, d \nu
&\le
\big\| \big( \sum_{i=1}^n |Tx_i|^p \big)^\frac 1p \big\|_{Y}
\big\| \big( \sum_{i=1}^n | y'_i|^s \big)^\frac 1s \big\|_{Y'}\\
&\le \rho_{p,q}(T)  \big\| \big( \sum_{i=1}^n |x_i|^q \big)^\frac 1q \big\|_{X}
\big\| \big( \sum_{i=1}^n | y'_i|^s \big)^\frac 1s \big\|_{Y'}.
\end{align*}
This gives (ii), with $C \le \rho_{p,q}(T).$

(ii) $\Rightarrow$ (i): Suppose there is $C>0$ such that
$$
\int \big( \sum_{i=1}^n |Tx_i y'_i|^r \big)^\frac 1r \, d \nu \le
C \big\| \big( \sum_{i=1}^n |x_i|^q \big)^\frac 1q \big\|_{X}
\big\| \big( \sum_{i=1}^n | y'_i|^s \big)^\frac 1s \big\|_{Y'}
$$
for every $\{x_i\}_{i=1}^n\subset X$ and $\{y'_i\}\subset Y'$. By Lemma \ref{dual}, it follows that
\begin{align*}
\big\| \big( \sum_{i=1}^n |Tx_i |^p \big)^\frac 1p \big\|_{Y}& =
\sup \Big\{ \big\| \big( \sum_{i=1}^n |Tx_i \cdot y'_i|^r \big)^\frac 1r \big\|_{L^1(\nu)}  :
\big\| \big( \sum_{i=1}^n |y'_i|^s \big)^\frac 1s \big\|_{Y'} \le 1 \Big\}\\
& \le
C \big\| \big( \sum_{i=1}^n |x_i|^q \big)^\frac 1q \big\|_{X} .
\end{align*}
This also gives that $\rho_{p,q} (T) \le C$.
\end{proof}

\section{Lattice tensor norms and duality for $(p,q)$-regular operators} \label{tensor}

In this section we introduce some specific topologies for tensor products of Banach lattices, following a similar  procedure as in the case of the Laprest\'e tensor norms \cite{deflo}. This approach allows us to relate $(p,q)$-regular operators with other classical operator ideals by using standard tensor product duality.

Let us start by recalling the definitions and notation of well-known tensor norms. Given Banach spaces $X$ and $Y$, for $z\in X\otimes Y$ and $1\leq p\leq \infty$ set
\begin{align*}
\varepsilon(z)&= \sup \Big\{ x^*\otimes y^*(z):x^*\in B_{X^*},\,y^*\in B_{Y^*} \Big\},\\
\pi(z)&= \underset{z= \sum x_i \otimes y_i }{\inf} \Big\{  \sum \|x_i\|\|y_i\| \Big\},\\
g_p(z)&= \underset{z= \sum x_i \otimes y_i }{\inf} \Big\{ \Big( \sum \|x_i\|^p \Big)^{1/p} \cdot \sup_{y^* \in B_{Y^*}} \Big( \sum | \langle y_i, y^* \rangle |^{p'} \Big)^{1/p'} \Big\},\\
d_p(z)&= \underset{z= \sum x_i \otimes y_i }{\inf} \Big\{  \sup_{x^* \in B_{X^*}} \Big( \sum | \langle x_i, x^* \rangle |^{p'} \Big)^{1/p'} \cdot \Big( \sum \|y_i\|^p \Big)^{1/p} \Big\},\\
w_p(z)&= \underset{z= \sum x_i \otimes y_i }{\inf} \Big\{ \sup_{x^* \in B_{X^*}} \Big( \sum | \langle x_i, x^* \rangle |^{p} \Big)^{1/p}  \cdot \sup_{y^* \in B_{Y^*}} \Big( \sum | \langle y_i, y^* \rangle |^{p'} \Big)^{1/p'} \Big\}.
\end{align*}
These are the injective, projective and some particular cases of the Laprest\'e tensor norms (Chevet-Saphar tensor norms, see \cite[12.5]{deflo}).

\begin{definition}
Let $X$, $Y$ be Banach lattices and $1 \le q \le p \le \infty$. For $z\in X \otimes Y$, let us define the positively homogeneous function
$$
\phi_{p,q}(z) :=
\inf \Big\{   \big\| \big( \sum |x_i|^{q} \big)^{1/q} \big\|_X \cdot
\big\| \big( \sum |y_i|^{p'} \big)^{1/p'} \big\|_Y : \, z= \sum_{i=1}^n  x_i \otimes y_i \Big\}
$$
and the seminorm
$$
r_{p,q}(z):=
\inf \Big\{ \sum_{j=1}^m   \phi_{p,q}(z_j): \, z= \sum_{j=1}^m z_j\Big\}.
$$
\end{definition}

\begin{proposition} \label{proppqreg}
Let $1 \le q \le p \le \infty$, $X$ and $Y$ Banach lattices.
\begin{itemize}
\item[(1)]For $z \in X \otimes Y$, $\varepsilon(z) \le r_{p,q}(z) \le \pi(z).$ Consequently, $r_{p,q}$ is a norm.
\item[(2)] If $X$ is $q$-convex and $Y$ is $p'$-convex (with $M^{(q)}(X)=1=M^{(p')}(Y)$), then  $\phi_{p,q}$ is a quasi-norm with constant $2^{1-t}$ where $1/t=1/q + 1/p'$. In particular, for $p=q$,
$r_{p,p}=\phi_{p,p}$ is a norm.
\end{itemize}
\end{proposition}

\begin{proof}
(1) For the inequality $\varepsilon(z) \le r_{p,q}(z)$ it is enough to prove that $\varepsilon(z) \le \phi_{p,q}(z)$ for every $z \in X \otimes Y.$ For every representation $z= \sum_{i=1}^n x_i \otimes y_i$ and $x^* \in B_{X^*}$, $y^* \in B_{Y^*}$, we have
\begin{align*}
\sum_{i=1}^n \langle x_i,x^* \rangle \langle y_i, y^* \rangle & \le (\sum_{i=1}^n  |\langle x_i,x^* \rangle|^p)^{1/p} \cdot(\sum_{i=1}^n  |\langle y_i,y^* \rangle |^{p'} )^{1/p'}\\
&\le (\sum_{i=1}^n  |\langle x_i,x^* \rangle|^q)^{1/q} \cdot (\sum_{i=1}^n  |\langle y_i,y^* \rangle |^{p'} )^{1/p'}\\
&\le \big\| \big( \sum |x_i|^{q} \big)^{1/q} \big\|_X \cdot \big\| \big( \sum |y_i|^{p'} \big)^{1/p'} \big\|_Y.
\end{align*}
This gives $\varepsilon(z) \le r_{p,q}(z)$. For $r_{p,q}(z) \le \pi(z)$, just note that for each representation $z =\sum_{j=1}^m x_j \otimes y_j$,
$$
r_{p,q}(z) \le \sum_{j=1}^m \phi_{p,q}(x_j \otimes y_j) \le \sum_{j=1}^m \|x_j\|_X \, \|y_j\|_Y.
$$

(2) Let $1/t=1/q + 1/p'$, and note that $t \le 1$. Given $z_1,z_2\in X\otimes Y$ and $\epsilon >0$, let $z_1 = \sum_{i=1}^n x^1_i \otimes y^1_i$ and $z_2 = \sum_{i=1}^n x^2_i \otimes y^2_i$ such that
$$
\big\|  \sum_{i=1}^n |x^1_i|^{q}  \big\|^{1/q}_{X_{[q]}}
=
\big\| \big( \sum_{i=1}^n |x^1_i|^{q} \big)^{1/q} \big\|_X
\le (\phi_{p,q}(z_1) +\epsilon)^{t/q},
$$
$$
\big\| \sum_{i=1}^n |y^1_i|^{p'}  \big\|^{1/p'}_{Y_{[p']}} =
\big\| \big( \sum_{i=1}^n |y^1_i|^{p'} \big)^{1/p'} \big\|_Y \le (\phi_{p,q}(z_1) +\epsilon)^{t/p'},
$$
and
$$
\big\|  \sum_{i=1}^n |x^2_i|^{q}  \big\|^{1/q}_{X_{[q]}}
=
\big\| \big( \sum_{i=1}^n |x^2_i|^{q} \big)^{1/q} \big\|_X
\le (\phi_{p,q}(z_2) +\epsilon)^{t/q},
$$
$$
\big\| \sum_{i=1}^n |y^2_i|^{p'}  \big\|^{1/p'}_{Y_{[p']}} =
\big\| \big( \sum_{i=1}^n |y^2_i|^{p'} \big)^{1/p'} \big\|_Y \le (\phi_{p,q}(z_2) +\epsilon)^{t/p'}.
$$
Then
\begin{align*}
\phi_{p,q}(z_1+z_2)& \le \bigg\| \big( \sum_{i=1}^n |x^1_i|^{q} + \sum_{i=1}^n |x^2_i|^{q}  \big)^{1/q} \bigg\|_X  \bigg\| \big( \sum_{i=1}^n |y^1_i|^{p'} +\sum_{i=1}^n |y^2_i|^{p'}  \big)^{1/p'} \bigg\|_Y\\
&= \bigg\|  \sum_{i=1}^n |x^1_i|^{q}   + \sum_{i=1}^n |x^2_i|^{q}  \bigg\|^{1/q}_{X_{[q]}}   \bigg\| \sum_{i=1}^n |y^1_i|^{p'}  +\sum_{i=1}^n |y^2_i|^{p'}  \bigg\|^{1/p'}_{Y_{[p']}}\\
&\le \Big(\big\|  \sum_{i=1}^n |x^1_i|^{q} \big\|_{X_{[q]}}    + \big\| \sum_{i=1}^n |x^2_i|^{q}  \big\|_{X_{[q]}}  \Big)^{1/q}  \Big(\big\| \sum_{i=1}^n |y^1_i|^{p'} \big\|_{Y_{[p']}}  +\| \sum_{i=1}^n |y^2_i|^{p'}  \big\|_{Y_{[p']}}  \Big)^{1/p'}\\
&\le\Big(  (\phi_{p,q}(z_1) +\epsilon)^{t} + (\phi_{p,q}(z_2) +\epsilon)^{t}  \Big)^{1/q}  \Big(  (\phi_{p,q}(z_1) +\epsilon)^{t} + (\phi_{p,q}(z_2) +\epsilon)^{t}  \Big)^{1/p'}\\
&\le \Big( 2^{1-t} (\phi_{p,q}(z_1) + \phi_{p,q}(z_2) + 2\epsilon) \Big)^{t/q}  \Big(2^{1-t}  (\phi_{p,q}(z_1) + \phi_{p,q}(z_2) + 2\epsilon) \Big)^{t/p'} \\
&= 2^{1-t} (\phi_{p,q}(z_1) + \phi_{p,q}(z_2) + 2\epsilon).
\end{align*}
As $\varepsilon>0$ was arbitrary, it follows that $\phi_{p,q}(z_1+z_2)\leq 2^{1-t}(\phi_{p,q}(z_1) + \phi_{p,q}(z_2))$ as claimed.
\end{proof}

Now we can provide the representation theorem for $(p,q)$-regular operators. Recall that the trace duallity allows us to identify $(X\otimes Y^*)^*$ with a certain subspace of $L(X,Y^{**})$: for $\varphi\in (X\otimes Y^*)^*$, take $T_\varphi:X\to Y^{**}$ given by
$$
\langle T_\varphi(x),y^*\rangle=\varphi(x\otimes y^*)
$$
for $x\in X$ and $y^*\in Y^*$.

\begin{theorem} \label{reppqreg}
Let $1 \le q \le p \le \infty$. Then
$$
\mathcal R_{p,q}(X,Y)= \Big( X \otimes_{r_{p,q}} Y^* \Big)^* \cap \mathcal L(X,Y).
$$
 isometrically.
\end{theorem}
\begin{proof}
To see the inclusion $\subseteq$ just take a $(p,q)$-regular operator $T:X \to Y$ and consider the trace duality with a tensor $z = \sum_{j=1}^m \sum_{i=1}^n x_i^j \otimes y_i^{j*} \in X \otimes Y^*$. We have by Proposition \ref{formulabil} for $r=1$ (and so $s=p'$) that
\begin{align*}
\langle T, z \rangle &= \sum_{j=1}^m \sum_{i=1}^n \langle T(x^j_i), y_i^{j*} \rangle \le \sum_{j=1}^m \sum_{i=1}^n | \langle T(x^j_i), y_i^{j*} \rangle |\\
&\le \sum_{j=1}^m  \rho_{p,q}(T) \, \big\| \big( \sum |x^j_i|^{q} \big)^{1/q} \big\|_X \cdot\big\| \big( \sum |y_i^{j*}|^{p'} \big)^{1/p'} \big\|_{Y^*}.
\end{align*}
Since this holds for all representations of $z$, it follows that the functional $\varphi_T$ defined by $T$ as $\varphi_T(x \otimes y^*)= \langle T(x), y^* \rangle$ satisfies that
$$
\| \varphi_T \| \le \rho_{p,q}(T).
$$

For the converse inclusion, take a functional $\varphi:  X \otimes_{r_{p,q}} Y^*  \to \mathbb R$ and define the operator $T_{\varphi}:X \to Y$ by $\langle T_\varphi(x), y^* \rangle = \varphi(x \otimes y^*)$. Let $\{x_i\}_{i=1}^n\subset X$.  For every $\epsilon >0$, there is a tensor $z =   \sum_{i=1}^n  x_i^* \otimes  y_i^{0*} $ with  $ \| (\sum_{i=1}^n|y_i^{0*}|^{p'} )^{\frac{1}{p'}} \|_{Y^*} \le 1$ such that
$$
\sup \bigg\{ \sum_{i=1}^n \langle T_\varphi(x_i), y_i^{*} \rangle : \, \big\|\big(\sum_{i=1}^n|y_i^*|^{p'}\big)^{\frac{1}{p'}}\big\|_{Y^*} \le 1 \bigg\}  \leq  \sum_{i=1}^n \langle T_\varphi(x_i), y_i^{0*} \rangle+\varepsilon.
$$
By Proposition \ref{formulabil} for $r=1$, it follows that
\begin{align*}
\bigg\|\bigg(\sum_{i=1}^n|T_{\varphi}(x_i)|^p\bigg)^{\frac1p}\bigg\| - \epsilon&\leq \sum_{i=1}^n \langle T_\varphi(x_i), y_i^{0*} \rangle
= \varphi \big( \sum_{i=1}^n x_i \otimes y_i^{0*}  \big)  \le \|\varphi\|  r_{p,q}(z)\\
&\le \| \varphi\| \big\| ( \sum_{i=1}^n |x_i|^q)^{1/q} \big\|_X \big\|(\sum_{i=1}^n|y_i^{0*}|^{p'})^{\frac{1}{p'}} \big\|_{Y^*}\\
&\le \| \varphi \|  \big\| ( \sum_{i=1}^n |x_i|^q)^{1/q} \big\|_X.
\end{align*}
Therefore, $T_\varphi$ is $(p,q)$-regular and $\rho_{p,q}(T_\varphi) \le \|\varphi \|.$ This finishes the proof.
\end{proof}

This approach fits actually with a more general framework, in which analogous tensor product representations are also possible for the case of $p$-convex and $p$-concave operators. In order to compare these with classical operator ideals, we introduce the following.

\begin{definition}
Let $1 \le q \le p \le \infty$, and $X$, $Y$ Banach lattices.
For $z\in X \otimes Y$, define the positively homogeneous functions
$$
\delta_{p,q}(z) := \inf \Big\{   \big(  \sum \|x_i\|_X^{q} \big)^{1/q}\big\| \big( \sum |y_i|^{p'} \big)^{1/p'} \big\|_Y : \, z= \sum_{i=1}^n  x_i^j \otimes y_i^j \Big\}
$$
$$
\iota_{p,q}(z) := \inf \Big\{   \big\| \big( \sum |x_i|^{q} \big)^{1/q} \big\|_X \big( \sum \|y_i\|_Y^{p'} \big)^{1/p'} : \, z= \sum_{i=1}^n  x_i^j \otimes y_i^j \Big\},
$$
and the corresponding seminorms
$$
h_{p,q}(z):=
\inf \Big\{ \sum_{j=1}^m   \delta_{p,q}(z_j): \, z= \sum_{j=1}^m z_j\Big\}.
$$
$$
k_{p,q}(z):=
\inf \Big\{ \sum_{j=1}^m   \iota_{p,q}(z_j): \, z= \sum_{j=1}^m z_j\Big\}.
$$
\end{definition}

Let us write $\mathcal{CX}_{p.q}(X,Y)$ for the space of $(p,q)$-convex operators from the Banach space $X$ to the Banach lattice $Y$, and $\mathcal{CC}_{p.q}(X,Y)$ for the space of $(p,q)$-concave operators from the Banach lattice $X$ to the Banach space $Y$. The following facts can be proved arguing as in Proposition \ref{proppqreg} and Theorem \ref{reppqreg}.

\begin{proposition} \label{proppqcon}
Let $1 \le q \le p \le \infty$, $X$ and $Y$ Banach lattices.
\begin{itemize}
\item[(1)]
For $z \in X \otimes Y$, $\varepsilon(z) \le h_{p,q}(z) \le \pi(z)$ and
$\varepsilon(z) \le k_{p,q}(z) \le \pi(z)$. Consequently, $h_{p,q}$ and $k_{p.q}$ are  norms.

\item[(2)] If $Y$ is $p'$-convex (constant 1),
then  $\delta_{p,q}$ is a quasi-norm with constant $2^{1-t}$ where $1/t=1/q + 1/p'$. In particular,
$h_{p,p}=\delta_{p,p}$ is a norm.

\item[(3)] If $X$ is $q$-convex (constant 1), then
$\iota_{p,q}$ is a quasi-norm with constant $2^{1-t}$ where $1/t=1/q + 1/p'$.  In particular,
$k_{p,p}=\iota_{p,p}.$
\end{itemize}
\end{proposition}

\begin{theorem} \label{reppqcom}
Let $1 \le q \le p \le \infty$. Then we have the following isometric identities:
\begin{itemize}
\item[(1)]
$
\mathcal{CX}_{p,q}(X,Y)= \Big( X \otimes_{h_{p,q}} Y^* \Big)^* \cap \mathcal L(X,Y),
$
\item[(2)]
$
\mathcal{CC}_{p,q}(X,Y)= \Big( X \otimes_{k_{p,q}} Y^* \Big)^* \cap \mathcal L(X,Y).
$

\end{itemize}
\end{theorem}

This point of view allows us to state the relations among $(p,q)$-regular, $(p,q)$-convex and $(p,q)$-concave operators in a straightforward way by comparing the norms appearing in the corresponding representations.

\begin{proposition}
Let $1 \le q \le p \le \infty$. Then we have the following relations.
\begin{itemize}

\item[(1)]
\begin{itemize}
\item[(i)] If $X$ is $q$-convex, then $\mathcal{R}_{p,q}(X,Y) \subseteq \mathcal{CX}_{p,q}(X,Y).$

\item[(ii)]  If $X$ is $q$-concave, then $\mathcal{CX}_{p,q}(X,Y) \subseteq \mathcal{R}_{p,q}(X,Y).$

\item[(iii)] If $X$ is an $L^q$-space, then $\mathcal{R}_{p,q}(X,Y) = \mathcal{CX}_{p,q}(X,Y).$

\end{itemize}

\vspace{2mm}

\item[(2)]\begin{itemize}
\item[(i)] If $Y$ is $p$-concave, then $\mathcal{R}_{p,q}(X,Y) \subseteq \mathcal{CC}_{p,q}(X,Y).$

\item[(ii)]  If $Y$ is $p$-convex, then $\mathcal{CC}_{p,q}(X,Y) \subseteq \mathcal{R}_{p,q}(X,Y).$

\item[(iii)] If $Y$ is an $L^p$-space, then $\mathcal{R}_{p,q}(X,Y) = \mathcal{CC}_{p,q}(X,Y).$

\end{itemize}

\vspace{2mm}

\item[(3)] $\mathcal L(L^q,L^p) = \mathcal{R}_{p,q}(L^q,L^p) = \mathcal{CC}_{p,q}(L^q,L^p) =\mathcal{CX}_{p,q}(L^q,L^p).$

\end{itemize}
\end{proposition}
\begin{proof}
The proofs of (1) and (2) are direct consequences of the previous results and direct duality arguments.
For the proof of (3) just note that $p' \le q'$ and for every representation of a tensor as $z= \sum_{i=1}^n x_i \otimes y_i^*$ we have
\begin{align*}
\pi(z) &\le \sum_{i=1}^n \|x_i\| \, \|y_i^*\| \le \Big( \sum_{i=1}^n \| x_i\|^q \Big)^{1/q} \Big( \sum_{i=1}^n \| y_i^*\|^{q'} \Big)^{1/q'}\\
&\le  \Big( \sum_{i=1}^n \| x_i\|^q \Big)^{1/q} \Big( \sum_{i=1}^n \| y_i^*\|^{p'} \Big)^{1/p'}.
\end{align*}
\end{proof}

Besides this, more can be said about the coincidence of $(p.q)$-regular operators between $L^r$-spaces. This will be addressed in Section \ref{subsecFacBan}.

Some direct consequences of the general theory of summability of Banach lattices can also be stated in this framework. Using the representation theorem for maximal operator ideals (see \cite[p.203]{deflo}) and the tensor norms associated to the ideals of $p$-summing and $p$-dominated operators, we obtain the following:

\begin{proposition}
Let $X, \,Y$ be Banach lattices and $1 \le p \le \infty$. The following relations hold:
\begin{itemize}
\item[(i)] $w_p \le r_{p,p}$ and $ \mathcal D_p(X,Y) \subseteq \mathcal R_{p,p}(X,Y),$ where $\mathcal D_p$ is the ideal of $p$-dominated operators.

\item[(ii)]  $g_p \le h_{p,p}$ and $ \Pi_{p'}^*(X,Y) \subseteq \mathcal{CX}_{p,p}(X,Y),$ where $\Pi_{p'}^*(X,Y)$ is the adjoint to the ideal of $p'$-summing operators.

\item[(iii)]  $d_p \le k_{p',p'}$ and so $\Pi_{p'}(X,Y) \subseteq \mathcal{CC}_{p',p'} (X,Y)$.

\end{itemize}
\end{proposition}
\begin{proof}
For (i), use $w_p=w_{p'}^t$ (\cite[p.152]{deflo}) and the fact that the ideal of $p'$-dominated operators is associated to the norm $w_p^*$. Then $(w^*_p)'= w_{p'}$. A direct calculation using that for $x_1,\ldots,x_n \in X$,
$$
\sup_{x^* \in B_{X^*}} \Big( \sum |\langle x_i, x^* \rangle|^p \Big)^{1/p} \le \Big\| \Big( \sum | x_i|^p \Big)^{1/p} \Big\|,
$$
 gives $w_{p'} \le r_{p',p'}.$ Thus the above comments give the required inclusion.

For (ii), use that $\Pi^*_{p'}$ is associated to $g_p'$ (\cite[p.211]{deflo} and \cite[\S.17.9]{deflo}). The inequality $g_p \le h_{p,p}$ is given by using the same inequality as in (i). Similar arguments
show (iii).
\end{proof}

In the particular case when we deal with $L_p$ spaces, the so called Chevet-Persson-Saphar inequalities (see \cite[15.10]{deflo}), provide a useful tool for relating the tensor norms we just introduced with other classical tensor norms:
$$
d^*_{p'}(z) \le d_p(z) \le \Delta_p(z) \le g^*_{p'}(z) \le g_p(z), \quad z \in L^p \otimes Y.
$$
Moreover, in the case when $Y$ is also an $L^p$-space, we actually get
\begin{equation} \label{chepersaeq}
d^*_{p'}(z) = d_p(z) = \Delta_p(z) = g^*_{p'}(z) = g_p(z), \quad z \in L^p(\mu) \otimes L^p(\nu)
\end{equation}
for arbitrary measures $\mu$ and $\nu$.

Let us now focus on the topological properties of the tensor product $L^p(\mu) \otimes L^{p'}(\nu)$ endowed with the $r_{p,p}$-norm, and compare them with other classical topologies. We will comment on $p=2$ and the general case separately:

\begin{itemize}

\item[(1)] For $p=2$, the equalities given in (\ref{chepersaeq}) show that $L^2(\mu) \otimes_\alpha L^{2}(\nu)$ cannot be isomorphic to $L^2(\mu) \otimes_{r_{2,2}} L^{2}(\nu)$ for $\alpha=d^*_{2}= d_2= \Delta_2 = g^*_{2} = g_2$, since by Corollary \ref{GroTh}, we have that $r_{2,2}$ is equivalent to $\pi$ in this tensor product.

\item[(2)] However, we can easily see that $r_{2,2}$ is equivalent to $d_2$ on the tensor product $X \otimes \ell^2$ for $X= \ell^\infty$ or $X= \ell^1$. This is a direct consequence of the so called Little Grothendieck Theorem (see \cite[17.14]{deflo}) and Corollary \ref{GroTh}.

\item[(3)] Similarly, $d_\infty$ is also equivalent to $r_{2,2}$, in this case as a consequence of Grothendieck Theorem (see \cite[17.14]{deflo}) and Corollary \ref{GroTh}.

\end{itemize}

For the general case $(p \ne 2)$, the Chevet-Persson-Saphar inequalities yield some positive results about $(s,q)$-regularity for certain well-known operators. The following are just a sample.

\begin{itemize}
\item[(1)] For every Banach lattice $Y$, $d^*_{p'} \le d_p \le \Delta_p \le g^*_{p'} \le g_p \le r_{s,q}$ for every $1 \le q \le p \le s$ in the tensor product $L^p(\mu) \otimes Y$. This is a direct consequence of the $p$-concavity of $L^p$.

\item[(2)]  The norm $\Delta_p$  concerns spaces of Bochner integrable functions, and can also be related to operators $T:L^p(\mu) \to Y$ defined as $Y$-valued integral by means of the formula
$f \mapsto \int \phi \, f \, d\mu \in Y$ for a certain function $\phi \in L^{p'}(\mu,Y) \hookrightarrow (L^p(\mu) \otimes_{\Delta_p} Y^*)^*$. Thus, the comparison of $\Delta_p$ and $r_{s,q}$ provides also some meaningful results.
\end{itemize}

Using the tensor product representation of maximal operator ideals as dual spaces of topological tensor products (see \cite[17.5]{deflo}), the above arguments yield the following.

\begin{corollary}
Let  $Y$ be a Banach lattice, $1 \le q \le r \le p$ and $T:L^p(\mu) \to Y$.
\begin{itemize}

\item[(1)] Every $r$-integral and $r$-summing operator from $L^r(\mu)$ to $Y$, as well as operators belonging to their associated dual ideals, are $(p,q)$-regular.

\item[(2)]  Every operator $T:L^r(\mu) \to Y$ defined as the $Y$-valued integral $T(\cdot)= \int \phi \, (\cdot) \, d\mu$ for a function $\phi \in L^{r'}(\mu,Y)$ is $(p,q)$-regular.
\end{itemize}
\end{corollary}

Let us recall that an equivalent form of Grothendieck's Theorem can be given in terms of tensor products of $C(K)$-spaces \cite[Theorem 3.1]{pisgro}: for each pair of
compact Hausdorff spaces $K_1$ and $K_2$ and every $z = \sum_{i=1}^n x_i \otimes y_i \in C(K_1) \otimes C(K_2)$, it holds that
$$
\pi(z) \le \, K_G \, \| (\sum_{i=1}^n |x_i|^2)^{1/2} \|_{C(K_1)} \cdot  \| (\sum_{i=1}^n |y_i|^2)^{1/2} \|_{C(K_2)},
$$
where $K_G$ is Grothendieck's constant.

Krivine's version of Grothendieck's Theorem (cf. \cite[Theorem 1.f.14]{LT2}) states that $\mathcal R_{2,2}(X,Y) = \mathcal L(X,Y)$, and the corresponding constants are related as $\|T\| \le \rho_{2,2}(T) \le K_G \|T\|.$ The following result is the pre-dual version of this fact, and so is also equivalent to Grothendieck's Theorem.

\begin{corollary} \label{GroTh}
Let $X$ and $Y$ be Banach lattices. Then $\pi \le K_G \, r_{2,2} \le K_G \, \pi$ on $X \otimes Y$.
\end{corollary}

\begin{proof}
Take a tensor $z=\sum_{i=1}^n x_i \otimes y_i \in X \otimes Y,$ and consider $x_0 = (\sum_{i=1}^n |x_i|^2)^{1/2} \in X$ and $y_0 =(\sum_{i=1}^n |y_i|^2)^{1/2} \in Y$. Take the ideals $I(x_0)$ and $I(y_0)$ generated
by these elements in the corresponding lattices, and endow both of them with $AM$-norms:
$$
\|x\|_{X,\infty}= \inf \Big\{ \lambda: \, |x| \le \lambda \frac{x_0}{\|x_0\|} \Big\} \,\,\,\,\,\, \textrm{and} \,\,\,\,\,\,
\|y\|_{Y,\infty}= \inf \Big\{ \lambda: \, |y| \le \lambda \frac{y_0}{\|y_0\|} \Big\}.
$$
Note that $\|x\|_X \le  \|x\|_{X,\infty}$ and $\|y\|_Y \le \|y\|_{Y,\infty}$, and so the inclusion maps $J_X: \overline I(x_0) \to X$ and $J_Y: \overline I(y_0) \to Y$ acting in the closure of these ideals satisfy that
$\|J_X\| \le 1$ and $\|J_Y\| \le 1,$ respectively. By Kakutani's theorem these can be considered as $C(K)$ spaces (cf. \cite[Theorem 1.b.6]{LT2}. Note also that
$$
\|x_0\|_{X,\infty} = \|x_0\|_{X} = \| (\sum_{i=1}^n |x_i|^2 )^{1/2}\|_X \,\,\,\, \textrm{and} \,\,\,\,\,\, \|y_0\|_{Y,\infty} = \|y_0\|_{Y} = \| (\sum_{i=1}^n |y_i|^2 )^{1/2}\|_Y.
$$

Since $\pi$ satisfies the metric mapping property (see for example \cite[3.2]{deflo}), we have that
$$
\big\| J_X \otimes J_Y: \overline I(x_0) \otimes_\pi  \overline I(y_0) \to X \otimes_\pi Y \big\| \le \|J_X\| \cdot \|J_Y\| \le 1.
$$

Now, we apply Grothendieck's Theorem for tensor products of $C(K)$ spaces to obtain
\begin{align*}
\pi(\sum_{i=1}^n x_i \otimes y_i)&= \pi(\sum_{i=1}^n J_X(x_i) \otimes J_Y(y_i) ) \le K_G \, \big\| (\sum_{i=1}^n |x_i|^2)^{1/2} \big\|_{X,\infty} \cdot \big\| ( \sum_{i=1}^n |y_i|^2)^{1/2} \big\|_{Y,\infty}\\
&= K_G \, \big\| (\sum_{i=1}^n |x_i|^2)^{1/2} \big\|_{X} \cdot \big\| ( \sum_{i=1}^n |y_i|^2)^{1/2} \big\|_{Y}.
\end{align*}
The result follows by convexity.
\end{proof}

\section{$(p,q)$-regular operators between $L_r$-spaces and Marcinkiewicz-Zygmund inequalities} \label{subsecFacBan}

In this section we will center our attention in the case of operators defined between $L_r$-spaces, in relation with the Marcinkiewicz-Zygmund type inequalities presented by A. Defant and M. Junge  in \cite{DJ}. By means of the so called Maurey-Rosenthal factorization theory (see for instance \cite{defa,defasan}), we will be able to extend these results to the case of operators acting in $r$-convex function lattices and with values in $r$-concave Banach function lattices. In particular, we study the requirements to reduce the study of $(p,q)$-regular operators between Banach lattices to the properties of such operators between $L_r$-spaces.

For the case $p=q$, a Maurey-Rosenthal factorization theorem for  $p$-regular operators holds under the usual convexity/concavity requirements. The following result is similar to Theorem 3.1 in \cite{defasan}. However, the reader must notice that the requirements on the operator $T$ are different. We sketch the proof showing the argument; see the proof of Theorem 3.2 in \cite{defasan} for more details.

\begin{theorem} \label{thA}
Let $1 \le s \le p <\infty$.
Let $X$ be a $p$-convex order continuous Banach function space over $(\Omega,\Sigma,\mu)$ and  $Y$ be an $s$-concave Banach function space over $(\Omega',\Sigma',\nu)$ with $Y'$ order continuous. Let $T:X \to Y$ be a $p$-regular operator. Then
$T$ factors as
$$
\xymatrix{
X \ar[rr]^{T} \ar@{.>}[d]_{M_f} & &   Y\\
L_{p}(\mu)  \ar@{.>}[rr]^{\hat T}&  & L_{s}(\nu)
\ar@{.>}[u]_{M_g}}
$$
for suitable functions $f$ and $g$. Here, $\hat T$ is a linear and continuous operator.
\end{theorem}

\begin{proof} Note first that $1/p+1/s' \le 1$. Take $r\geq 1$ such that $1/r=1/p+1/s'$.
For the aim of simplicity we assume that the $p$-convexity and $s$-concavity constants involved are equal to $1$; note that, being $s$-concave, $Y$ is order continuous. For $\{x_i\}_{i=1}^n\subset  X$ and $\{y'_i\}_{i=1}^n\subset Y'$, using the generalized H\"older inequality \eqref{eq:genholder} and the fact that $T$ is $p$-regular, it holds that
\begin{align*}
\Big( \sum_{i=1}^n \|Tx_i y'_i \|_{L_1(\nu)}^r \Big)^\frac 1r &\le \int \Big( \sum_{i=1}^n |Tx_i |^p \Big)^\frac 1p  \Big( \sum_{i=1}^n | y'_i |^{s'} \Big)^\frac{1}{s'} \, d \nu\\
& \le C \big\| \big( \sum_{i=1}^n |x_i|^p \big)^\frac 1p \big\|_{X}\big\| \big( \sum_{i=1}^n | y'_i|^{s'} \big)^\frac {1}{s'} \big\|_{Y'}.
\end{align*}
Since $X$ is $p$-convex and $Y'$ is $s'$-convex, by Young's inequality, it follows that
\begin{align*}
\sum_{i=1}^n \|Tx_i y'_i \|_{L_1(\nu)}^r &\le C^{r} \big\| \sum_{i=1}^n |x_i|^p  \big\|_{X_{[p]}}^{\frac rp} \big\|  \sum_{i=1}^n | y'_i|^{s'}  \big\|^{\frac{r}{s'}}_{(Y')_{[s']}}\\
&\le C^{r} \Big( \frac{r}{p} \big\| \sum_{i=1}^n |x_i|^p  \big\|_{X_{[p]}} + \frac{r}{s'}\big\|  \sum_{i=1}^n | y'_i|^{s'}  \big\|_{(Y')_{[s']}} \Big).
\end{align*}

A  standard application of Ky Fan's Lemma provides functions $f_0 \in (X(\mu)_{[p]})'=\mathcal M(X(\mu),L^p(\mu))_{[p]}$ (where $\mathcal M$ denotes the space of multiplication operators),  and
$g_0 \in \big( (Y(\nu)')_{[s']} \big)'=\mathcal M(Y(\nu)',L^{s'}(\nu))_{[s']}$ such that
$$
\int |Tx \, y' | d \nu \le C \Big( \int |x|^p f_0 \, d \mu \Big)^{1/p} \, \Big( \int |y'|^{s'} g_0 \, d \nu \Big)^{1/{s'}}.
$$
This gives the inequality
$$
\Big( \int \Big|\frac{Tx}{g_0^{1/s'}} \Big|^{s} d \nu \Big)^{1/s} \le C \Big( \int |x|^p f_0 \, d \mu \Big)^{1/p}
$$
for all $x \in X(\mu)$. This provides the desired factorization, using the associated multiplication operators and $T$ to define the operator $\hat T$ (see the argument given in \cite[Th.3.2]{defasan}).  The decomposition of the operator is given by
$$
M_{g_0^{1/s'}} \circ g_0^{-1/s'}T(\cdot/f_0^{1/p}) \circ M_{f_0^{1/p}} =T,
$$
that is, $f=f_0^{1/p}$, $\hat T= g_0^{-1/s'}T(\cdot/f_0^{1/p})$ and $g=g_0^{1/s'}$.
\end{proof}

\begin{remark}
The statement of the previous result excludes the fundamental case in which $Y=L_1$, since the dual of such space is not order continuous. It must be mentioned here that this was one of the most relevant instances of the original factorization of Maurey and provides some of its main applications, for example regarding the structure of reflexive subspaces of $L^1$-spaces (see II.H.13 in \cite{wojt}). We will not consider this case for the aim of simplicity. However, the result is expected to be true also in this case, since the separation argument can be extended for this case using a nowadays well-known procedure that is explained in \cite{defa}.
\end{remark}

\begin{remark}
The result above assures that the operator factors through an operator $\hat T:L_p(\mu) \to L_s(\nu)$ for $1 < s \le p < \infty$. However, as can be seen in \cite[Corollary on page 282]{DJ}, not every operator from $L_p$ to $L_s$ is $p$-regular.
In fact, regarding operators between $L_p$-spaces, we trivially have the following: Every operator $T:L_r(\mu) \to L_t(\nu)$, for $1 \le r \le t \le \infty$, is $(t,r)$-regular.

Indeed, let $\{x_i\}_{i=1}^n\subset  L_r(\mu)$.
\begin{align*}
\Big\| \big( \sum_{i=1}^n |Tx_i|^t \big)^{1/t} \Big\|_{L_t(\nu)}& =  \big( \sum_{i=1}^n \big\|Tx_i \big\|_{L_t(\nu)}^t \big)^{1/t} \le \|T\| \big( \sum_{i=1}^n \big\|x_i \big\|_{L_r(\nu)}^t \big)^{1/t}\\
&\le \|T\| \big( \sum_{i=1}^n \big\|x_i \big\|_{L_r(\nu)}^r \big)^{1/r} =  \|T\| \Big\| \big( \sum_{i=1}^n |x_i|^r \big)^{1/r} \Big\|_{L_r(\nu)}.
\end{align*}
In the rest of this section we will analyze what happens when $r=t$ ($s=p$ above).
\end{remark}

Theorem \ref{thA} does not give a priori any information on the regularity properties of $\hat{T}$; of course, $T$ need not be positive, otherwise the results regarding this property are trivial.
Next result characterizes the existence of a factorization for $(p,q)$-regular operators through $L_r$-spaces preserving the property of being $(p,q)$-regular. The requirements are formally more restrictive than the ones needed in Theorem \ref{thA}, and they involve a new vector  norm inequality that suggests some mixed norms for Banach function spaces.

\begin{theorem}\label{t:factor2}
Let $X$ be an $r$-convex Banach function space over $(\Omega,\Sigma,\mu)$, and let $Y$ be an $r$-concave Banach function space over $(\Omega',\Sigma',\nu)$ such that $X$ and $Y'$ are order continuous. Let $T:X \to Y$ be an operator. The following are equivalent.

\begin{itemize}

\item[(i)] There is a constant $K>0$ such that for each pair of matrices of elements $(x_{i,j})_{i=1,j=1}^{n,m}$ and $(y'_{i,j})_{i=1,j=1}^{n,m}$ in $X$ and $Y'$, respectively, the following inequality holds.
$$
\sum_{i=1}^n \sum_{j=1}^m \big|\big\langle Tx_{i,j} \, y'_{i,j} \big\rangle\big| \le K \,\Big\|  \big( \sum_{i=1}^n \big( \sum_{j=1}^m |x_{i,j}|^q \big)^\frac{r}{q} \big)^\frac1r \Big\|_X
\,\Big\|  \big( \sum_{i=1}^n \big( \sum_{j=1}^m |y'_{i,j}|^{p'} \big)^\frac{r'}{p'} \big)^\frac{1}{r'} \Big\|_{Y'}
$$

\item[(ii)] There is a constant $K>0$ such that for each matrix  of elements $(x_{i,j})_{i=1,j=1}^{n,m}$ in $X$, respectively, the following inequality holds.
$$
\Big\|  \big( \sum_{i=1}^n \big( \sum_{j=1}^m |T x_{i,j}|^p \big)^\frac{r}{p} \big)^\frac1r \Big\|_Y
\le K  \Big\| \big( \sum_{i=1}^n \big( \sum_{j=1}^m |x_{i,j}|^q \big)^\frac{r}{q} \big)^\frac1r \Big\|_X
$$

\item[(iii)]  There are functions $f$ and $g$ such that $T$ factors as

$$
\xymatrix{
X \ar[rr]^{T} \ar@{.>}[d]_{M_f} & &   Y\\
L_r(\mu)  \ar@{.>}[rr]^{\hat T}&  & L_r(\nu)
\ar@{.>}[u]_{M_g}}
$$

where $\hat T$ is $(p,q)$-regular.
\end{itemize}

\end{theorem}
\begin{proof} It can be seen as a consequence of the lemmata of the introductory sections that (i) and (ii) are equivalent. Let us prove that
(i) $\Rightarrow$ (iii). The argument is similar to the one in the proof of theorem \ref{thA}. Since $Y'$ is also $r'$-convex (with constant $1$), and writting the inequalities in (i) in the form
$$
 \sum_{i,j=1}^{n,m}  \big\| Tx_{i,j} \, y'_{i,j} \big\|_{L_1(\nu)} - K \frac rp \,\Big\|  \sum_{i=1}^n \big( \sum_{j=1}^m |x_{i,j}|^q \big)^\frac{r}{q}  \Big\|_{X_{[r]}}
$$
$$
 - K \frac rq
\,\Big\|   \sum_{i=1}^n \big( \sum_{j=1}^m |y'_{i,j}|^{p'} \big)^\frac{r'}{p'}  \Big\|_{(Y')_{[r']}} \le 0,
$$
 we can define a set of functions $\phi:B_{X_{[r]}} \times B_{(Y')_{[r]}} \to \mathbb R$ as
\begin{align*}
\phi(f,g):=  &\sum_{i=1}^n \sum_{j=1}^m \big\| Tx_{i,j} \, y'_{i,j} \big\|_{L_1(\nu)} -  \,K \frac rp \int \big( \sum_{i=1}^n \big( \sum_{j=1}^m |x_{i,j}|^q \big)^\frac{r}{q} \big) \, f \, d \mu\\
&-\, K \frac rq \int \big( \sum_{i=1}^n \big( \sum_{j=1}^m |y'_{i,j}|^{p'} \big)^\frac{r'}{p'} \big) \,g \, d \nu.
\end{align*}
These functions are continuous when the product of the weak* topologies are defined in the  product $B_{X_{[r]}} \times B_{(Y')_{[r']}}$, and convex; the family of all the functions (defined for each finite couple of matrices $(x_{i,j})_{i=1,j=1}^{n,m}$ and $(y'_{i,j})_{i=1,j=1}^{n,m}$) is concave, and so Ky Fan's Lemma we find two functions $f_0 \in B_{X_{[r]}} $ and $g_0 \in B_{(Y')_{[r']}}$ such that
$$
\big\| \sum_{j=1}^m | Tx_{j} \, y'_{j} | \big\|_{L_1(\nu)} \le  \,K \frac rp \int \big( \sum_{j=1}^m |x_{i,j}|^q \big)^\frac{r}{q}  \, f_0 \, d \mu+
\, K \frac rq \int \big( \sum_{j=1}^m |y'_{i,j}|^{p'} \big)^\frac{r'}{p'} \,g_0 \, d \nu.
$$
Using a standard argument based on the homogeneity of this expression, we get the inequality
$$
\big\| \sum_{j=1}^m | Tx_{j} \, y'_{j} | \big\|_{L_1(\nu)} \le  \,K  \Big( \int \big( \sum_{j=1}^m |x_{i,j}|^q \big)^\frac{r}{q}  \, f_0 \, d \mu \Big)^\frac 1r
\cdot \Big( \int \big( \sum_{j=1}^m |y'_{i,j}|^{p'} \big)^\frac{r'}{p'} \,g_0 \, d \nu \Big)^\frac 1{r'}.
$$
A standard argument already used in this paper based on Lemma \ref{dual} allow to find the inequality
$$
\Big\| \big( \sum_{j=1}^m | \frac{Tx_{j}}{g_1} |^p \big)^\frac 1p \Big\|_{L_r(\nu)} \le K  \Big( \int \big( \sum_{j=1}^m |x_{i,j}|^q \big)^\frac{r}{q}  \, f_1 \, d \mu \Big)^\frac 1r =
K  \Big\| \big( \sum_{j=1}^m |x_{i,j}|^q \big)^\frac{1}{q} f_1 \Big\|_{L_r(\mu)}
$$
for functions $f_1$ and $g_1$ defined as the $1/r$-powers of $f_0$ and $g_0$. This gives the factorization ---when the inequality is considered just for $m=1$ and all possible vectors $x$---, and the $(p,q)$-regularity of operator $\hat T$ from $L_r(\mu)$ to $L_r(\nu)$.

The converse is straightforward, and so the proof is finished.
\end{proof}

A factorization for $T$ as the one given in (iii) of Theorem \ref{t:factor2} is usually called a \textit{strong factorization of $T$ through $L^r$-spaces.}

The characterization of Marcinkiewicz-Zygmund type inequalities given in \cite{DJ} provides the following.

\begin{corollary}\label{c:MZ}
Assume that $1 \le r_1,r_2,p,q \le \infty$. Then $\mathcal{R}_{p,q}(L_{r_1},L_{r_2}) =L(L_{r_1},L_{r_2}) $ in the following cases.

\begin{itemize}

\item If $q \le r_1=r_2 \le p$.

\item If $r_1=r_2 =1 $ and $1 \le q \le p \le \infty$.

\item If $r_1=r_2 = \infty$ and $1 \le q \le p \le \infty$.

\item If $1 \le r_2 \le r_1 < 2$ and there exists  $t$ such that $q \le t \le p$ and   $r_1 < t \le 2$.

\item  If $2 < r_2 \le r_1 \le \infty$ and there exists $t$ such that $q \le t \le p$ and  $2 \le  t < r_2$.

\item If $1 \le r_2 \le 2 \le r_1 \le \infty$ and $q \le 2 \le p$.

\end{itemize}

\end{corollary}

For  $r_1=r_2=r$ and using Theorem  \ref{t:factor2}, we obtain the following result.

\begin{corollary} \label{primpqreg}
Assume that $1 \le p,q \le \infty$. Let $1 <r < \infty$ and let $X$ be an $r$-convex Banach function space over $(\Omega,\Sigma,\mu)$, and $Y$ an $r$-concave Banach function space over $(\Omega',\Sigma',\nu)$ such that $X$ and $Y'$ are order continuous. Let $T:X \to Y$ be an operator.  Suppose that
$$
[q,p] \cap [\min\{r,2\}, \max\{r,2\}] \neq \emptyset.
$$

\vspace{0.3cm}

Then the  following assertions are equivalent.

\begin{itemize}

\item[(i)] There is a constant $K>0$ such that for each matrix  of elements $(x_{i,j})_{i=1,j=1}^{n,m}$ in $X$, respectively, the following inequality holds.
$$
\Big\|  \big( \sum_{i=1}^n \big( \sum_{j=1}^m |T x_{i,j}|^p \big)^\frac{r}{p} \big)^\frac1r \Big\|_Y
\le K  \Big\| \big( \sum_{i=1}^n \big( \sum_{j=1}^m |x_{i,j}|^q \big)^\frac{r}{q} \big)^\frac1r \Big\|_X.
$$

\item[(ii)]  There  is a strong factorization of $T$ through $L_r$-spaces.
\end{itemize}

\end{corollary}

Together with Theorem \ref{thA}, this corollary provides and equivalence among $(p,q)$-regular-type properties for the operator $T$, that may be understood as
a generalization of  Marcinkiewicz-Zygmund inequalities. Let us finish the section with this result.

\begin{corollary}
Under the assumptions on $p,q,r$, $X$ and $Y$ in Corollary \ref{primpqreg}, with
$$
[q,p] \cap [\min\{r,2\}, \max\{r,2\}] \neq \emptyset.
$$

The following statements are equivalent.

\begin{itemize}

\item[(i)] The operator $T:X\to Y$ is $r$-regular.

\item[(ii)] There is a constant $K>0$ such that for each matrix  of elements $(x_{i,j})_{i=1,j=1}^{n,m}$ in $X$, respectively, the following inequality holds.
$$
\Big\|  \big( \sum_{i=1}^n \big( \sum_{j=1}^m |T x_{i,j}|^p \big)^\frac{r}{p} \big)^\frac1r \Big\|_Y
\le K  \Big\| \big( \sum_{i=1}^n \big( \sum_{j=1}^m |x_{i,j}|^q \big)^\frac{r}{q} \big)^\frac1r \Big\|_X.
$$

\end{itemize}

\end{corollary}

\section{Extension properties}

The definition of $(p,q)$-regular operator makes also sense for operators define d on a subspace of a Banach lattice: Given $X$, $Y$ Banach lattices, and $X_0\subset X$ a closed subspace, an operator $T:X_0\rightarrow Y$ is $(p,q)$-regular if there is $K>0$ such that
$$
\bigg\|\bigg(\sum_{i=1}^n|Tx_i|^p\bigg)^{\frac1p}\bigg\|\leq K \bigg\|\bigg(\sum_{i=1}^n|x_i|^q\bigg)^{\frac1q}\bigg\|,
$$
for every $(x_i)_{i=1}^n\subset X_0$ (with the obvious modification when $p$ or $q$ are infinite).

In \cite[Theorem 4]{Pisier} it is shown that every $\infty$-regular operator defined on a closed subspace of a Banach lattice with values in another Banach lattice extends to a $\infty$-regular operator on the whole Banach lattice. This extension property also holds for $(\infty,q)$-regular operators as the following shows.

Before the proof, let us recall the Calder\'on product construction (cf. \cite{calderon}): for Banach lattices $X_0$, $X_1$ and $\theta\in(0,1)$ let $X_0^{1-\theta}X_1^\theta$ the space of elements $f$ for which there exist $f_0\in X_0$, $f_1\in X_1$ such that $|f|\leq |f_0|^{1-\theta}|f_1|^\theta$ and let
$$
\|f\|_{X_0^{1-\theta}X_1^\theta}=\inf\{ \|f_0\|^{1-\theta}\|f_1\|^\theta:|f|\leq |f_0|^{1-\theta}|f_1|^\theta,\,\textrm{with }f_i\in X_i\}.
$$
This expression defines a norm at least for the class of spaces which are of interest here (see \cite{calderon}).

\begin{lemma}\label{l:extension}
Let $1\leq q\leq \infty$, let $X$ be a $q$-convex Banach lattice of measurable functions on $(\Omega,\Sigma,\mu)$ (with $q$-convexity constant equal to one) and let $X_0\subset X$ be a closed subspace. Every $(\infty,q)$-regular operator $T:X_0\rightarrow \ell_q^n$ has an $(\infty,q)$-regular extension $\tilde T:X\rightarrow \ell_q^n$ with $\rho_{\infty,q}(\tilde T)\leq \rho_{\infty,q}(T)$.
\end{lemma}

\begin{proof}
We  follow a similar approach to that of \cite[Theorem 4]{Pisier}. Let $Z$ be the tensor product $\ell_{q'}^n\otimes X$, where $\frac1q+\frac1{q'}=1$, endowed with
$$
\|v\|_Z=\inf\Big\{\Big(\sum_{i=1}^n|a_i|^{q'}\Big)^{\frac{1}{q'}}\Big\|\Big(\sum_{i=1}^n |x_i|^q\Big)^{\frac1q}\Big\|_X:\, v=\sum_{i=1}^n a_ie_i\otimes x_i\Big\}.
$$
Let us see that $\|\cdot\|_Z$ indeed defines a norm.

Let $E_0$ be  the space of $n$-tuples of functions of $L_\infty(\mu)$ endowed with the norm
$$
\| (g_i)_{i=1}^n\|_{E_0}=\sum_{i=1}^n \|g_i\|_\infty.
$$

Let $E_1$ be the space of $n$-tuples of measurable functions $(h_i)_{i=1}^n\subset L_0(\mu)$ such that $|h_i|^{\frac1q}\in X$ for $i=1,\ldots,n$, equipped with
$$
\| (h_i)_{i=1}^n\|_{E_1}=\Big\|\Big(\sum_{i=1}^n |h_i|\Big)^{\frac1q}\Big\|_{X}^{q}.
$$
This is indeed a norm since $X$ is $q$-convex with constant $1$.

We claim that $\|\cdot\|_Z$ coincides with the norm of the space $E_0^{1-\theta}E_1^\theta$ for $\theta=\frac1q$ under the identification mapping $(f_i)_{i=1}^n\in E_0^{1-\theta}E_1^\theta$ to $\sum_{i=1}^n e_i\otimes f_i\in\ell_{q'}^n\otimes X$.

Indeed, note first that for $1\leq i\leq n$, there exist $g_i\in L_\infty(\mu)$ and $h_i\in L_0(\mu)$ with $|h_i|^\theta=|h_i|^{\frac1q}\in X$, such that $|f_i|\leq |g_i|^{1-\theta}|h_i|^\theta$, which yield that $f_i\in X$ for $1\leq i\leq n$, and in particular $\sum_{i=1}^n e_i\otimes f_i\in\ell_{q'}^n\otimes X$. Now, given $\varepsilon>0$, let $(g_i)_{i=1}^n$ and $(h_i)_{i=1}^n$ as above so that
$$
\|f\|_{E_0^{1-\theta}E_1^\theta}\geq\Big(\sum_{i=1}^n\|g_i\|_\infty\Big)^{\frac{1}{q'}}\Big\|\Big(\sum_{i=1}^n|h_i|\Big)^{\frac1q}\Big\|_X-\varepsilon.
$$
Set $a_i=\|g_i\|_\infty^{\frac{1}{q'}}$, and $A_i=\{\omega\in \Omega: g_i(\omega)\neq 0\}$. Since $f_i=f_i\chi_{A_i}$, it follows that
$$
|h_i|\geq\frac{|f_i|^q}{|g_i|^{\frac{q}{q'}}}\chi_{A_i}.
$$
Hence,
$$
\Big\|\Big(\sum_{i=1}^n|h_i|\Big)^{\frac1q}\Big\|_X\geq\Big\|\Big(\sum_{i=1}^n\frac{|f_i|^q}{|g_i|^{\frac{q}{q'}}}\chi_{A_i}\Big)^{\frac1q}\Big\|_X\geq \Big\|\Big(\sum_{i=1}^n\frac{|f_i|^q}{a_i^{q}}\Big)^{\frac1q}\Big\|_X.
$$
Thus,
$$
\|f\|_{E_0^{1-\theta}E_1^\theta}+\varepsilon\geq\Big(\sum_{i=1}^na_i^{q'}\Big)^{\frac{1}{q'}}\Big\|\Big(\sum_{i=1}^n\Big|\frac{f_i}{a_i}\Big|^{q}\Big)^{\frac1q}\Big\|_X\geq \Big\|\sum_{i=1}^n a_ie_i\otimes \frac{f_i}{a_i}\Big\|_Z=\Big\|\sum_{i=1}^n e_i\otimes f_i\Big\|_Z,
$$
and as $\varepsilon>0$ is arbitrary, the inequality 
$$
\|f\|_{E_0^{1-\theta}E_1^\theta}\geq \Big\|\sum_{i=1}^n e_i\otimes f_i\Big\|_Z
$$ 
follows. For the converse inequality, let $\varepsilon>0$ and $f=\sum_{i=1}^n a_ie_i\otimes x_i$ with 
$$
\|f\|_Z+\varepsilon\geq \Big(\sum_{i=1}^n |a_i|^{q'}\Big)^{\frac{1}{q'}}\Big\|\Big(\sum_{i=1}^n|x_i|^{q}\Big)^{\frac1q}\Big\|_X.
$$
Let $h_i=|x_i|^q$ and $g_i=|a_i|^{q'}\chi_\Omega$. Hence, as $|f_i|=|a_ix_i|\leq (|a_i|^{q'})^{\frac{1}{q'}}(|h_i|)^{\frac1q}=|g_i|^{1-\theta}|h_i|^\theta$, it follows that
$$
\|f\|_Z+\varepsilon\geq \Big(\sum_{i=1}^n \|g_i\|_\infty\Big)^{\frac1{q'}}\Big\|\Big(\sum_{i=1}^n|h_i|\Big)^{\frac1q}\Big\|_X\geq \| (g_i)_{i=1}^n\|_{E_0}^{1-\theta}\|(h_i)_{i=1}^n\|_{E_1}^\theta\geq\|f\|_{E_0^{1-\theta}E_1^\theta}.
$$
As $\varepsilon>0$ is arbitrary, the claim follows.

Now, we claim that $Z^*=R_{\infty,q}(X,\ell_q^n)$ isometrically. Indeed, given $u:X\rightarrow \ell_q^n$ we have
\begin{align*}
\rho_{\infty,q}(u)&=\sup\bigg\{\Big\|\bigvee_{i=1}^m|u x_i|\Big\|_{\ell_q^n}: \Big\|\Big(\sum_{i=1}^m|x_i|^q\Big)^{\frac1q}\Big\|_X\leq1,\,  m\in\mathbb N\bigg\}\\
&=\sup\bigg\{\Big\|\sum_{k=1}^n\Big(\bigvee_{i=1}^m|\langle e_k^*,u x_i\rangle|\Big)e_k\Big\|_{\ell_q^n}: \Big\|\Big(\sum_{i=1}^m|x_i|^q\Big)^{\frac1q}\Big\|_X\leq1,\,  m\in\mathbb N\bigg\}\\
&=\sup\bigg\{\sum_{k=1}^n a_k\bigvee_{i=1}^m|\langle e_k^*,u x_i\rangle|: \Big\|\sum_{k=1}^n a_k e_k^*\Big\|_{\ell_{q'}^n},\,\Big\|\Big(\sum_{i=1}^m|x_i|^q\Big)^{\frac1q}\Big\|_X\leq1,\,  m\in\mathbb N\bigg\}\\
&=\sup\bigg\{\Big|\sum_{k=1}^n a_k \langle e_k^*,u x_{i_k} \rangle\Big|:\Big(\sum_{k=1}^n |a_k|^{q'} \Big)^{\frac{1}{q'}},\, \Big\|\Big(\sum_{k=1}^n |x_{i_k}|^q\Big)^{\frac1q}\Big\|_X\leq1,  (i_k)_{k=1}^n\subset[1,m]\bigg\}\\
&= \sup\{|\langle u,v\rangle|:\|v\|_Z\leq1\}=\|u\|_{Z^*}.
\end{align*}

Finally, consider the subspace $M\subset Z$ formed by all $v=\sum_{k=1}^n a_k e_k\otimes x_k$ such that $x_k\in X_0$ for $k=1,\ldots,n$. Let $T:X_0\rightarrow Y$ be a $(\infty,q)$-regular operator. Given $v\in M$, and $\varepsilon>0$, take scalars $a_k$ and $x_k\in X_0$ such that
$$
\Big(\sum_{k=1}^n|a_k|^{q'}\Big)^{\frac{1}{q'}}\Big\|\Big(\sum_{k=1}^n|x_k|^q\Big)^{\frac1q}\Big\|\leq \|v\|_Z+\varepsilon.
$$
By \cite[Proposition 1.d.2.]{LT2}, we have
\begin{align*}
|\langle T,v\rangle|&=\Big|\sum_{k=1}^n\langle T(x_k),a_k e_k\rangle\Big|\\
&\leq \langle\bigvee_{k=1}^n|T(x_k)|,\sum_{k=1}^n|a_k | e_k \rangle\\
&\leq \Big\|\bigvee_{k=1}^n|T(x_k)|\Big\|_{\ell_q^n}\Big\|\sum_{k=1}^n|a_k| e_k \Big\|_{\ell^n_{q'}}\\
&\leq\rho_{\infty,q}(T)(\|v\|_Z+\varepsilon).
\end{align*}

Since $\varepsilon>0$ is arbitrary, it follows that for every $v\in M$ we get $|\langle T,v\rangle| \leq\rho_{\infty,q}(T)$. Hence, we can consider a Hahn-Banach extension of $v\in M\mapsto \langle T,v\rangle$ with norm not exceeding $\rho_{\infty,q}(T)$. This extension is clearly of the form $v\in Z\mapsto\langle \tilde T,v\rangle$ for some operator $\tilde T:X\rightarrow \ell_q^n$ with the required properties.
\end{proof}

\begin{theorem}\label{t:extension}
Let $1\leq q\leq \infty$, and measure spaces $(\Omega,\Sigma,\mu)$, $(\Omega',\Sigma',\nu)$. Given a closed subspace $X_0\subset L_q(\mu)$ and an $(\infty,q)$-regular operator $T:X_0\rightarrow L_q(\nu)$, there is an $(\infty,q)$-regular extension $\tilde{T}:L_q(\mu)\rightarrow L_q(\nu)$ with $\rho_{\infty,q}(\tilde T)=\rho_{\infty,q}(T)$.
\end{theorem}

\begin{proof}
For simplicity, assume $\Omega=\Omega'=[0,1]$ endowed with Lebesgue measure. The proof can be easily carried over to general measure spaces. For every $n\in \mathbb N$, let $P_n:L_q\rightarrow \ell_q^{2^n}$ be given by
$$
P_n f=2^{\frac{n}{q'}}\sum_{i=1}^{2^n}\int_{\frac{i-1}{2^n}}^{\frac{i}{2^n}} f d\nu \cdot e_i.
$$
Notice that the norm of $P_n$ is less or equal than one. Also, let $J_n:\ell_q^{2^n}\rightarrow L_q$ be given by
$$
J_n e_i=2^{\frac{n}{q}}\chi_{[\frac{i-1}{2^n},\frac{i}{2^n}]},
$$
for $1\leq i\leq 2^n$ and extended linearly.  Its norm is equal to one: in fact it is an isometry.

Given a closed subspace $X_0\subset L_q(\mu)$ and an $(\infty,q)$-regular operator $T:X_0\rightarrow L_q(\nu)$, for each $n\in\mathbb N$, consider $T_n=P_n T:X_0\to \ell_q^{2^n}$. Since $P_n$ is positive, we clearly have $\rho_{\infty,q}(T_n)\leq\rho_{\infty,q}(T)$, so by Lemma \ref{l:extension}, there is an extension $\tilde T_n:L_q\to \ell_q^{2^n}$ with $\rho_{\infty,q}(\tilde T_n)\leq\rho_{\infty,q}(T)$.

Given $f\in L_q(\mu)$, the sequence $(J_n\tilde T_n f)_{n\in\mathbb N}$ is bounded in $L_q$-norm. By reflexivity, it has a weakly convergent subsequence. Let $\mathcal U$ be a free ultrafilter in $\mathbb N$ and for each $f\in L_q(\mu)$, let
$$
\tilde T f=\lim_{n\in \mathcal U} J_n\tilde T_n f,
$$
the limit taken in the weak topology along the ultrafilter $\mathcal U$.

It is clear that $\tilde T$ defines a bounded linear operator on $L_q$. We claim that $\tilde T$ is the required extension.

Indeed, for $f\in X_0$, note that
$$
J_n\tilde T_n f= J_n P_nT f\underset{n\to\infty}{\longrightarrow} Tf,
$$
which implies in particular that $\tilde T f=Tf$ for every $f\in X_0$.

Moreover, for every $(f_i)_{i=1}^m$ and every $g\in L_{q'}$ with $\|g\|_{q'}\leq 1$ we have
\begin{eqnarray*}
\Big\langle \bigvee_{i=1}^m|\tilde T f_i|, g\Big\rangle &=&\lim_{n\in\mathcal U}\Big\langle\bigvee_{i=1}^m |J_n\tilde T_n(f_i)|,g\Big\rangle\\
&\leq&\lim_{n\in\mathcal U}\Big\|\bigvee_{i=1}^m |J_n\tilde T_n(f_i)| \Big\|_q\\
&\leq&\lim_{n\in\mathcal U}\Big\|J_n\Big(\bigvee_{i=1}^m |\tilde T_n(f_i)|\Big)\Big\|_q\\
&\leq&\lim_{n\in\mathcal U}\Big\|\bigvee_{i=1}^m |\tilde T_n(f_i) | \Big\|_q\\
&\leq& \rho_{\infty,q}(T)\Big\|\Big(\sum_{i=1}^m|f_i|^q\Big)^{\frac1q}\Big\|.
\end{eqnarray*}
Thus,
$$
\Big\|\bigvee_{i=1}^m|\tilde T f_i|\Big\|\leq \rho_{\infty,q}(T)\Big\|\Big(\sum_{i=1}^m|f_i|^q\Big)^{\frac1q}\Big\|.
$$
\end{proof}

Note that by Corollary \ref{c:MZ}, every operator $T:L_q\to L_q$ is $(\infty,q)$-regular, while the projection onto the spam of the Rademacher sequence in $L_q$ is not $\infty$-regular.


\begin{thebibliography}{99}


\bibitem{AB}
	\textsc{C.~D. Aliprantis and O.~Burkinshaw,}
	\emph{Positive operators},
	Springer, Dordrecht, 2006, Reprint of the 1985 original.

\bibitem{B1}
	\textsc{A. V. Bukhvalov,}
	\textit{On complex interpolation method in spaces of vector-functions and generalized Besov spaces.}
	Dokl. Akad. Nauk SSSR \textbf{260} (1981), no. 2, 265--269.

\bibitem{B2}
	\textsc{A. V. Bukhvalov,}
	\textit{Order-bounded operators in vector lattices and spaces of measurable functions.}
	Translated in J. Soviet Math. \textbf{54} (1991), no. 5, 1131--1176. Itogi Nauki i Tekhniki, Mathematical analysis, Vol. 26 (Russian), 3--63, 148, Akad. Nauk SSSR, Vsesoyuz. Inst. Nauchn. i Tekhn. Inform., Moscow, 1988.


\bibitem{calderon}
	\textsc{A. P. Calder\'on,}
	\textit{Intermediate spaces and interpolation, the complex method.}
	Studia Math. \textbf{24} (1964), 113--190.

\bibitem{Danet}
	\textsc{N. Danet,}
	\textit{Lattice $(p,q)$-summing operators and their conjugates.}
	Stud. Cerc. Mat. \textbf{40} (1988), no. 1, 99--107.

\bibitem{defa}
 	\textsc{A. Defant,}
	\textit{Variants of the Maurey-Rosenthal theorem for quasi K\"othe function spaces},
	Positivity, \textbf{5} (2001) 153--175.

\bibitem{deflo}
	\textsc{A. Defant and K. Floret,}
	\emph{Tensor norms and operator ideals.}
	North-Holland Mathematics Studies, 176. North-Holland Publishing Co., Amsterdam, 1993.

\bibitem{defasan}
	\textsc{A. Defant and E.A. S\'anchez P\'erez,}
	 \textit{Maurey-Rosenthal factorization of positive operators and convexity.}
 	J. Math. Anal. Appl. \textbf{297} 2 (2004), 771--790.

\bibitem{DJ}
	\textsc{A. Defant and M. Junge,}
	\textit{Best constants and asymptotics of Marcinkiewicz-Zygmund inequalities.}
	Studia Math. \textbf{125} (1997), no. 3, 271--287.

\bibitem{GM}
	\textsc{J. Gasch and L. Maligranda,}
	\textit{On Vector-valued Inequalities of the Marcinkiewicz-Zygmund, Herz and Krivine Type.}
	Math. Nachr. \textbf{167} (1994), 95--129.
	
\bibitem{DJT}
	\textsc{J. Diestel, H. Jarchow, and A. Tonge,}
	\emph{Absolutely summing operators.}
	Cambridge Studies in Advanced Mathematics, 43. Cambridge University Press, 1995.

\bibitem{Fremlin}
	\textsc{D. H. Fremlin,}
	\textit{Tensor products of Banach lattices.}
	Math. Ann. \textbf{211} (1974), 87--106.
	
\bibitem{kalton}
	\textsc{N. J. Kalton}
	\textit{Convexity conditions for non-locally convex lattices.}
	Glasgow Math. J. \textbf{25} (1984), 141--152.

\bibitem{Krivine}
	\textsc{J. L. Krivine,}
	\textit{Th\`eor\`emes de factorisation dans les espaces r\`eticul\`es.}
	S\`eminaire Maurey-Schwartz 1973--1974: \textit{Espaces $L\sp{p}$, applications radonifiantes et g\`eom\`etrie des espaces de Banach,} Exp. Nos. 22 et 23. Centre de Math., \`Ecole Polytech., Paris, (1974).

\bibitem{Kusraev}
	\textsc{A. G. Kusraev,}
	\emph{Dominated operators.}
	Mathematics and its Applications, 519. Kluwer Academic Publishers, Dordrecht, 2000.
	
\bibitem{LT2}
	\textsc{J. Lindenstrauss and L. Tzafriri},
	\textit{Classical Banach Spaces II: Function Spaces.}
	Springer-Verlag, (1979).


\bibitem{NS}
	\textsc{N. J. Nielsen and J. Szulga,}
	\textit{$p$-lattice summing operators.}
	Math. Nachr. \textbf{119} (1984), 219--230.

\bibitem{Pisier}
	\textsc{G. Pisier,}
	\textit{Complex interpolation and regular operators between Banach lattices.}
	Arch. Math. (Basel) \textbf{62} (1994), no. 3, 261--269.

\bibitem{pisgro}
	\textsc{G. Pisier,}
	\textit{Grothendieck's Theorem, past and present},
	Bulletin of the American Mathematical Society  \textbf{49}, 2 (2012) 237--323.

\bibitem{Popa}
	\textsc{N. Popa,}
	\textit{Uniqueness of the symmetric structure in $L_p(\mu)$ for $0 < p < 1$.}
	Rev. Roum. Math. Pures et Appl. \textbf{27} (1982), 1061--1083.
	
\bibitem{RT2}
	\textsc{Y. Raynaud and P. Tradacete,}
	\textit{On $(p,q)$-regular operators and Calder\'on-Lozanovskii interpolation of quasi-Banach lattices.}
	Banach J. Math. Anal. (to appear).

\bibitem{Schep}
	\textsc{A. R. Schep,}
	\textit{Products and factors of Banach function spaces.}
	Positivity \textbf{14} (2010), 301--319.

\bibitem{wojt}
	\textsc{P. Wojtaszczyk,}
	\emph{Banach spaces for analysts.}
	Vol. 25. Cambridge University Press, Cambridge, 1996.

\end{thebibliography}
\end{document}